\documentclass[11pt,a4paper]{article}

\usepackage{header}
\usepackage{abbrev}

\begin{document}

\title{Riemannian gradient descent for spherical area-preserving mappings}

\author{Marco Sutti\thanks{Mathematics Division, National Center for Theoretical Sciences, Taipei, Taiwan (\email{msutti@ncts.tw}).}\hspace{2mm}\orcidlink{0000-0002-8410-1372} and Mei-Heng Yueh\thanks{Department of Mathematics, National Taiwan Normal University, Taipei, Taiwan (\email{yue@ntnu.edu.tw}).}\hspace{2mm}\orcidlink{0000-0002-6873-5818}}

\date{\today}

\maketitle

\begin{abstract}
We propose a new Riemannian gradient descent method for computing spherical area-preserving mappings of topological spheres using a Riemannian retraction-based framework with theoretically guaranteed convergence.
The objective function is based on the stretch energy functional, and the minimization is constrained on a power manifold of unit spheres embedded in 3-dimensional Euclidean space.
Numerical experiments on several mesh models demonstrate the accuracy and stability of the proposed framework. Comparisons with two existing state-of-the-art methods for computing area-preserving mappings demonstrate that our algorithm is both competitive and more efficient. Finally, we present a concrete application to the problem of landmark-aligned surface registration of two brain models.



\end{abstract}

\section{Introduction}

This paper proposes a new method for computing spherical area-preserving mappings of topological spheres using a Riemannian optimization framework.

Area-preserving mappings can serve as parameterization of surfaces and have been applied to the field of computer graphics and medical imaging~\cite{CCR:2020}.
State-of-the-art methods for the computation of area-preserving parameterizations include diffusion-based methods~\cite{CR:2018}, optimal mass transportation-based methods~\cite{Su:2016}, and nonlinear optimization methods~\cite{Yueh:2023}. It is worth noting that most previous works consider area-preserving parameterization from simply connected open surfaces to planar regions. 
For the spherical area-preserving mapping of genus-zero closed surfaces, recently developed methods include the adaptive area-preserving parameterization method~\cite{CGK:2022} and the stretch energy minimization method~\cite{Yueh:2019}.

Riemannian optimization generalizes unconstrained, continuous optimization (defined on Euclidean space) following the same principle that Riemannian geometry is a generalization of Euclidean space. In traditional optimization methods, nonlinear matrix manifolds rely on the Euclidean vector space structure. In contrast, in the Riemannian optimization framework, algorithms and convergence analysis are formulated based on the language of differential geometry, and numerical linear algebra techniques are then used for implementation. Riemannian optimization focuses on matrix manifolds, which are manifolds in the sense of classical Riemannian geometry that admit a natural representation of their elements in the form of matrices. Retraction mappings play a vital role in this framework, as they are used to turn objective functions defined in Euclidean space into functions defined on an abstract manifold so that constraints are taken into account explicitly.
A vast body of literature on the theory and implementation of Riemannian optimization now exists.
Some of the earliest references for this field, especially for line-search methods, go back to Luenberger~\protect{\cite{Luenberger:1973}}, Gabay~\protect{\cite{Gabay:1982}} and Udrişte~\protect{\cite{Udr:1994}}. More recent literature that encompasses the findings and developments of the last three decades is~\protect{\cite{EAS:1998,AMS:2008,Boumal:2023}}.

Our Riemannian gradient descent (RGD) method involves classical components from computational geometry (simplicial surfaces and mappings) and Riemannian optimization (line search, retractions, and Riemannian gradients). To the best of the authors' knowledge, this is the first time this approach has been used in computational geometry.

In this paper, we minimize the authalic energy for simplicial mappings $f\colon \cM\to\R^{3}$ defined as \protect{\cite{Yueh:2023}}
\[
   E_A(f) = E_S(f) - \cA(f),
\]
where $E_{S}$ is the stretch energy defined as
\[
   E_{S}(f) = \frac{1}{2} \, \vecop(\f)\tr (I_3\otimes L_S(f)) \vecop(\f),
\]
$\mathcal{A}(f)$ denotes the image area of the mapping $f$, 
subject to the constraint that the image of the mapping belongs to a power manifold of $n$ unit spheres embedded in $\R^{3}$, i.e.,
\[
   \Stwon = \underbrace{\Stwo \times \Stwo \times \cdots \Stwo}_{n \ \text{times}}.
\]

\subsection{Contributions}

In this paper, we propose an RGD method to compute spherical area-preserving mappings on the unit sphere. In particular, the main contributions of this paper are as follows:
\begin{enumerate}[(i)]
    \item We combine the tools from the Riemannian optimization framework and the components from computational geometry to propose a RGD method for computing spherical area-preserving mappings of topological spheres.
    \item We explore two different line-search strategies: one using MATLAB's \texttt{fminbnd}, while the other using the quadratic/cubic interpolant approximation from \protect{\cite[\S 6.3.2]{DennisSchnabel:1996}}.
    \item We conduct extensive numerical experiments on several mesh models to demonstrate the accuracy and efficiency of the algorithm.
    \item We demonstrate the competitiveness and efficiency of our algorithm over two state-of-the-art methods for computing area-preserving mappings.
    \item We show that the algorithm is stable with respect to small perturbations in the initial mesh model.
    \item We apply the algorithm to the concrete application of landmark-aligned surface registration between two human brain models.
\end{enumerate}

\subsection{Outline of the paper}

The remaining part of this paper is organized as follows. Section~\ref{sec:simplicial_objective} introduces the main concepts on simplicial surfaces and mappings, and presents the formulation of the objective function. Section~\ref{sec:riem_optim} provides some preliminaries on the Riemannian optimization framework. Section~\ref{sec:geometry} briefly recalls the geometry of the unit sphere and then describes the tools needed to perform optimization on the power manifold.
Section~\ref{sec:Riem_grad_descent} is about the RGD method and reports known convergence results.
Section~\ref{sec:numerical_experiments} discusses the extensive numerical experiments to evaluate our algorithm in terms of accuracy and efficiency, and provides a concrete application.
Finally, we wrap our paper with conclusions and future outlook in Section~\ref{sec:conclusions}.
The appendices contain other details about the calculations of the area of the simplicial mapping, the Hessian of the objective function, and the line-search procedure used in the RGD method.

\subsection{Notation}\label{sec:notation}

In this section, we list the notations and symbols adopted in order of appearance in the paper. Symbols specific to a particular section are usually not included in this list.

\begin{table}[htbp]
   \begin{center}
      \begin{tabular}{ll}
          $ \tau $        &  A triangular face \\
          $ \left\vert \tau \right\vert $  &  The area of the triangle $ \tau $ \\
          $ \cM $         &  Simplicial surface \\
          $ \mathcal{V}(\cM) $ & Set of vertices of $\cM$ \\
          $ \mathcal{F}(\cM) $ & Set of faces of $\cM$ \\
          $ \mathcal{E}(\cM) $ & Set of edges of $\cM$ \\
          $ v_{i}, \ v_{j}, \ v_{k} $      &  Vertices of a triangular face \\
          $ f $           &  Simplicial mapping \\
          $ \f $          &  Representative matrix of $f$ \\
          $ \f_{\ell} $   &  Coordinates of a vertex $ f(v_{\ell}) $ \\
          $ \vecop $      &  Column-stacking vectorization operator \\
          $ \Stwo $       &  Unit sphere in $ \R^{3} $ \\
          $ \Stwon $      &  Power manifold of $n$ unit spheres in $ \R^{3} $ \\
          $ E_{A}(f) $    &  The authalic energy \\
          $ E_{S}(f) $    &  The stretch energy \\
          $ L_{S}(f) $    &  Weighted Laplacian matrix \\
          $ \omega_{S} $  &  Modified cotangent weights \\
          $ \cA(f) $      &  Area of the image of $f$ \\
          $ \mathrm{T}_{\x}\Stwo $    &  Tangent space to $\Stwo$ at $\x$ \\
          $ \P_{\mathrm{T}_{\x}\Stwo} $    & Orthogonal projector onto the tangent space to $\Stwo$ at $\x$ \\
          $ \P_{\mathrm{T}_{\f_{\ell}}\Stwon} $  & Orthogonal projector onto the tangent space to $\Stwon$ at $\f_{\ell}$ \\
          $\Retraction$  &  Retraction mapping \\
          $\nabla E(f)$  &  Euclidean gradient of $E(f)$ \\
          $\grad E(f)$   &  Riemannian gradient of $E(f)$
      \end{tabular}
   \end{center}
\end{table}

\section{Simplicial surfaces, mappings, and objective function} \label{sec:simplicial_objective}

We introduce the simplicial surfaces and mappings in subsection \ref{sec:simplicial_surfaces} and the objective function in subsection \ref{sec:objective_function}.

\subsection{Simplicial surfaces and mappings} \label{sec:simplicial_surfaces}

A simplicial surface $\cM$ is the underlying set of a simplicial $2$-complex $\mathcal{K}(\cM) = \cF(\cM)\cup\cE(\cM)\cup\mathcal{V}(\cM)$ composed of vertices
$$
\mathcal{V}(\cM) = \left\{v_\ell = \left( v_\ell^1, v_\ell^2, v_\ell^2 \right)\tr \in\R^3 \right\}_{\ell=1}^n,
$$
oriented triangular faces
$$
\cF(\cM) = \left\{ \tau_\ell = [v_{i_\ell}, v_{j_\ell}, v_{k_\ell}] \mid v_{i_\ell}, v_{j_\ell}, v_{k_\ell} \in\mathcal{V}(\cM) \right\}_{\ell=1}^m,
$$
and undirected edges
$$
\cE(\cM) = \left\{ [v_i,v_j] \mid [v_i,v_j,v_k]\in\cF(\cM) \text{ for some $v_k\in\mathcal{V}(\cM)$} \right\}.
$$
A simplicial mapping $f\colon \cM\to\R^3$ is a particular type of piecewise affine mapping with the restriction mapping $f|_\tau$ being affine, for every $\tau\in\cF(\cM)$.
We denote 
$$
\f_\ell \coloneqq f(v_\ell) = \left( f_\ell^1, f_\ell^2, f_\ell^3 \right)\tr, ~ \textrm{ for every $v_\ell\in\mathcal{V}(\cM)$}.
$$
The mapping $f$ can be represented as a matrix
\begin{equation}\label{eq:representative_matrix}
    \f 
= \begin{bmatrix}
\f_1\tr \\
\vdots \\
\f_n\tr
\end{bmatrix}
= \begin{bmatrix}
f_1^1 & f_1^2 & f_1^3 \\
\vdots & \vdots & \vdots \\
f_n^1 & f_n^2 & f_n^3
\end{bmatrix}
=: \begin{bmatrix}
\f^1 & \f^2 & \f^3
\end{bmatrix},
\end{equation}
or a vector
$$
\vecop(\f) = \begin{bmatrix}
\f^1 \\
\f^2 \\
\f^3
\end{bmatrix}.
$$

\subsection{The objective function}\label{sec:objective_function}

The authalic energy for simplicial mappings $f\colon \cM\to\R^{3}$ is defined as~\protect{\cite{Yueh:2023}}
\[
   E_A(f) = E_S(f) - \cA(f),
\]
where $E_{S}$ is the stretch energy defined as
$$
   E_{S}(f) = \frac{1}{2} \, \vecop(\f)\tr (I_3\otimes L_S(f)) \vecop(\f),
$$
with $\otimes$ being the Kronecker product.
The weighted Laplacian matrix $ L_S(f) $ is given by
\begin{equation} \label{eq:L_S}
   [L_S(f)]_{i,j} =
\begin{cases}
-\sum_{[v_i,v_j,v_k]\in\cF(\cM)} [\omega_S(f)]_{i,j,k}  &\mbox{if $[{v}_i,{v}_j]\in\cE(\cM)$,}\\
-\sum_{\ell\neq i} [L_S(f)]_{i,\ell} &\mbox{if $j = i$,}\\
0 &\mbox{otherwise,}
\end{cases}
\end{equation}
in which $\omega_S(f)$ is the modified cotangent weight defined as
\begin{equation} \label{eq:omega}
   [\omega_S(f)]_{i,j,k} = \frac{\cot(\theta_{i,j}^k(f))  \, |f([v_i,v_j,v_k])|}{2|[v_i,v_j,v_k]|},
\end{equation}
with $\theta_{i,j}^k(f)$ being the angle opposite to the edge $f([v_i,v_j])$ at the point $f(v_k)$ on the image $f(\cM)$, as illustrated in Figure \ref{fig:cot}.

\begin{figure}[htbp]
\centering
\begin{tikzpicture}[thick,scale=1.2]
\coordinate (v_i) at (0,0);
\coordinate (v_j) at (0,2);
\coordinate (v_k) at (2,1);
\coordinate (v_l) at (-2,1);
\filldraw[green!20] (v_i) -- (v_j) -- (v_k);
\filldraw[green!20] (v_i) -- (v_j) -- (v_l);
\pic[draw, ->, "$\theta_{i,j}^k(f)$", angle eccentricity=2.05, angle radius=0.6cm]{angle = v_i--v_l--v_j};
\pic[draw, ->, "$\theta_{j,i}^\ell(f)$", angle eccentricity=2.05, angle radius=0.6cm]{angle = v_j--v_k--v_i};
\draw{
(v_i) -- (v_j) -- (v_k) -- (v_i) -- (v_l) -- (v_j)
};
\tikzstyle{every node}=[circle, draw, fill=yellow!20, inner sep=1pt, minimum width=2pt]
\draw{
(0,0) node{$\f_i$}
(0,2) node{$\f_j$}
(2,1) node{$\f_\ell$}
(-2,1) node{$\f_k$}
};
\end{tikzpicture}
\caption{An illustration of the cotangent weight defined on the image of $f$.}
\label{fig:cot}
\end{figure}
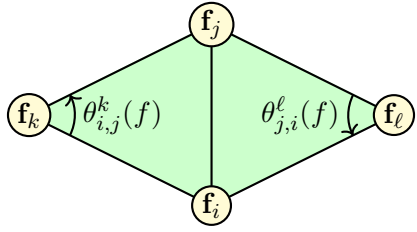

It is proved that $E_A(f)\geq 0$ and the equality holds if and only if $f$ preserves the area~\protect{\cite[Corollary 3.4]{Yueh:2023}}. 
Due to the optimization process, the image area $\cA(f)$ is not constant, hence we introduce a prefactor $|\cM|/\cA(f)$ and define the \emph{normalized stretch energy} as
\begin{equation}\label{eq:objective_function}
   E(f) = \frac{|\cM|}{\cA(f)} E_S(f).    
\end{equation}
To perform numerical optimization via the RGD method, we first need to compute the Euclidean gradient. By applying the formula $\nabla_{\f^s} E_S(f) = 2 L_S(f) \, \f^s$ from \cite[(3.6)]{Yueh:2023}, the gradient of $E(f)$ can be formulated as
\begin{align}\label{eq:egrad_obj_fun}
\nabla_{\f^s} E(f) &= \nabla_{\f^s} \left(\frac{|\cM|}{\cA(f)} E_S(f) \right) \nonumber\\
&= \frac{|\cM|}{\cA(f)} \nabla_{\f^s} E_S(f) + E_S(f) \nabla_{\f^s} \frac{|\cM|}{\cA(f)} \nonumber\\
&= \frac{2|\cM|}{\cA(f)} L_S(f) \,\f^s - \frac{|\cM| E_S(f)}{\cA(f)^2} \nabla_{\f^s} \cA(f) \nonumber\\
&= \frac{2|\cM|}{\cA(f)} L_S(f) \,\f^s - \left( \frac{|\cM| E_S(f)}{\cA(f)^2} \right) \nabla_{\f^s} \cA(f).
\end{align}
The following proposition gives an explicit formula for the calculation of $\nabla_{\f^s} \cA(f)$.
\begin{proposition}[Formula for $\nabla\cA$]
The gradient of $\cA$ can be explicitly formulated as
\begin{equation} \label{eq:GradA}
\nabla\cA(f|_\tau) = \frac{|\tau|}{\cA(f|_\tau)} \, \vecop(L_S(f|_\tau) \,\f_\tau).
\end{equation} 
\end{proposition}
\begin{proof}
By applying explicit formulae $E_S(f|_\tau) = \frac{\cA(f|_\tau)^2}{|\tau|}$ and $\nabla E_S(f|_\tau) = 2 \,\vecop(L_S(f|_\tau) \,\f_\tau)$ in \cite{Yueh:2023}, the chain rule yields
$$
2 \,\vecop(L_S(f|_\tau) \,\f_\tau) = \nabla E_S(f|_\tau) = \frac{2\cA(f|_\tau)}{|\tau|} \nabla \cA(f|_\tau),
$$
which is equivalent to \eqref{eq:GradA}.
\end{proof}
Other details about the calculation of $\cA(f)$ are reported in appendix~\ref{app:calculation_area}.

\section{Riemannian optimization framework} \label{sec:riem_optim}

The \textit{Riemannian optimization framework} \protect{\cite{EAS:1998,AMS:2008,Boumal:2023}} solves constrained optimization problems where the constraints have a geometric nature. This approach utilizes the underlying geometric structure of the problems, which allows the constraints to be taken explicitly into account. The optimization variables are constrained to a smooth manifold, and the optimization is performed on that manifold. In particular, in this paper, the problem is formulated on a power manifold of $n$ unit spheres embedded in $\R^{3}$, and we use the RGD method for minimizing the cost function~\eqref{eq:objective_function} on this power manifold.
Typically, the manifolds considered are matrix manifolds. This name is due to the fact that there is a natural representation of their elements in matrix form. Traditional methods for nonlinear matrix manifold optimization rely on the Euclidean vector space structure. 
For instance, in Euclidean space $\R^{n}$, the steepest descent method updates a current iterate $\x_{k}$ by moving in the direction $\dk$ of the anti-gradient, by a step size $\alpha_{k}$ chosen according to an appropriate line-search rule.
However, on a nonlinear manifold, the vector addition $ \x_{k} + \alpha_{k} \dk $ does not make sense in general due to the manifold curvature. Hence, we need the notion of tangent vectors and their length in order to generalize the steepest descent direction to a Riemannian manifold.
Riemannian optimization allows for a flexible and theoretically sound framework, using the language of differential geometry to formulate algorithms and perform the convergence analysis, while numerical linear algebra techniques are used for implementation. Concepts such as tangent spaces, projectors, and retraction mappings are key components of this framework.
Similarly to the line-search method in Euclidean space, a line-search method in the Riemannian framework determines at a current iterate $\x_{k} $ on a manifold $M$ a search direction $\boldxi$ on the tangent space $\TxM$. The next iterate $\x_{k+1}$ is then determined by a line search along a curve $\alpha \mapsto \Retraction_{\x}(\alpha \boldxi)$ where $\Retraction_{x} \colon \TxM \to M$ is called the \emph{retraction}.
The procedure is then repeated for $\x_{k+1}$ taking the role of $\x_{k}$. Search directions can be the negative of the Riemannian gradient $\grad f(\x)$ leading to the Riemannian steepest descent method. Other search directions lead to other methods, e.g., Riemannian versions of the trust-region method \protect{\cite{ABG:2007}} or BFGS \protect{\cite{RingWirth:2012}}.

The retraction is a vital tool in this framework. It is a mapping from the tangent space to the manifold, used to transform objective functions defined in Euclidean space into functions defined on a manifold while explicitly taking the constraints into account.

In the next section, we introduce some fundamental geometry concepts used in Riemannian optimization, which are necessary to formulate the RGD method for our problems.

\section{Geometry}\label{sec:geometry}

This section focuses on the tools from Riemannian geometry needed to generalize the classical steepest descent method to matrix manifolds.
We briefly recall the geometry of the unit sphere $\Stwo$ embedded in $\R^{3}$, and then we switch to the power manifold $\Stwon$.

\subsection{Unit sphere}\label{sec:unit_sphere_geometry}

The unit sphere $ \Stwo $ is a Riemannian submanifold of $ \R^{3} $ defined as
\[
   \Stwo = \lbrace \x \in \R^{3} \colon \x\tr \x = 1 \rbrace.
\]
The Riemannian metric (inner product) on the sphere is inherited from the embedding space $ \R^{3} $, i.e.,
\[
   \langle \boldxi, \boldsymbol{\eta} \rangle_{\x} = \boldxi\tr \boldsymbol{\eta}, \quad \boldxi, \, \boldsymbol{\eta} \in \TxS,
\]
where $\TxS$ is the tangent space to $ \Stwo $ at $ \x \in \Stwo $, defined as the set of all vectors orthogonal to $ \x $ in $ \R^{3} $, i.e.,
\[
   \TxS = \lbrace \z \in \R^{3} \colon \x\tr \z = 0 \rbrace.
\]
The projector $\P_{\TxS} \colon \R^{3} \to \TxS $ onto the tangent space $\TxS$ is defined by
\begin{equation}\label{eq:proj_on_TxS}
    \P_{\TxS}(\z) = (I_{3}-\x\x\tr)\,\z.
\end{equation}
In the following, points on the unit sphere are denoted by $\f_{i}$, and tangent vectors are represented by $\boldxi_{i}$.

\subsection{The power manifold \texorpdfstring{$\Stwon$}{TEXT}{}}\label{sec:power_manifold}

We aim to minimize a function $E(\f_{1},\dots,\f_{n})$ \eqref{eq:objective_function}, where each $\f_{i}$, $i=1,\dots,n$, lives on the same manifold $\Stwo$. This leads us to consider the \emph{power manifold} of $n$ unit spheres
\[
   \Stwon = \underbrace{\Stwo \times \Stwo \times \cdots \Stwo}_{n \ \text{times}},
\]
with the metric of $\Stwo$ extended elementwise. In the remaining part of this section, we present the tools from Riemannian geometry needed to generalize gradient descent to this manifold. The projector onto the tangent space is used to compute the Riemannian gradient. The retraction is used to turn an objective function defined on $\R^{n \times 3}$ into an objective function defined on the manifold $\Stwon$.

\subsubsection{Projection onto the tangent space}

As seen above, the projector onto the tangent space to the unit sphere at the point $ \x $ is given by~\eqref{eq:proj_on_TxS}.
Here, the points are denoted by $\f_{i} \in \R^{3} $, $i=1,\dots,n$, so we write
\[
   \P_{\mathrm{T}_{\f_{i}}\Stwo} = I_{3} - \f_{i} \f_{i}\tr.
\]
It clearly changes for every $ \f_{i} $. The projector onto the tangent space to the power manifold $\Stwon$ is a mapping
\[
   \PT \colon \R^{n \times 3} \to \mathrm{T}_{\f}\Stwon,
\]
and can be represented by a block diagonal matrix of size $3n \times 3n$, i.e.,
\begin{equation}\label{eq:projector_power_manifold}
   \PT \coloneqq \blkdiag\!\left(\P_{\mathrm{T}_{\f_{1}}\Stwo}, \P_{\mathrm{T}_{\f_{2}}\Stwo},\dots, \P_{\mathrm{T}_{\f_{n}}\Stwo} \right) = \begin{bmatrix}
       \P_{\mathrm{T}_{\f_{1}}\Stwo}  &  &  &  \\
         &  \P_{\mathrm{T}_{\f_{2}}\Stwo}  &  &  \\
         &  & \ddots & \\
         &  &  &  \P_{\mathrm{T}_{\f_{n}}\Stwo}
   \end{bmatrix}.
\end{equation}
When writing code, we never actually create this matrix. Instead, we implement an efficient version using vectorized operations (\verb|bsxfun|).

\subsubsection{Projection onto \texorpdfstring{$\Stwon$}{TEXT}{}}

The projection of a single point $\f_{i}$ from $ \R^{3} $ to the unit sphere $ \Stwo $ is given by the normalization as
\[
   \widetilde{\f}_{i} = \dfrac{\f_{i}}{\|\f_{i}\|_{2}}.
\]
The projection of the whole of $\f$ onto the power manifold $\Stwon$ is given by
\[
   \P_{\Stwon} \colon \R^{n \times 3} \to \Stwon,
\]
defined by
\[
   \f \mapsto \widetilde{\f} \coloneqq
   \diag\!\left( \dfrac{1}{\|\f_{1}\|_{2}}, \dfrac{1}{\|\f_{2}\|_{2}}, \ldots, \dfrac{1}{\|\f_{n}\|_{2}} \right) \begin{bmatrix} \f_{1} & \f_{2} & \cdots & \f_{n} \end{bmatrix} \tr.
\]
Again, this representative matrix is only shown for illustrative purposes; in the actual implementation, we use row-wise normalization of $\f$.

\subsubsection{Retraction onto \texorpdfstring{$\Stwon$}{TEXT}{}}\label{sec:retraction_on_S2n}

A retraction is a mapping from the tangent space to the manifold used to turn tangent vectors into points on the manifold and functions defined on the manifold into functions defined on the tangent space; see \protect{\cite{AMS:2008,Absil:2012}} for more details.

The retraction of a single tangent vector $ \boldxi_{i} $ from $ \mathrm{T}_{\f_{i}} \Stwo $ to $ \Stwo $ is a mapping $ \Retraction_{\f_{i}} \colon \mathrm{T}_{\f_{i}} \Stwo \to \Stwo $, defined by~\protect{\cite[Example 4.1.1]{AMS:2008}}
\[
   \Retraction_{\f_{i}}(\boldxi_{i}) = \dfrac{\f_{i} + \boldxi_{i}}{\|\f_{i} + \boldxi_{i} \|}.
\]
Figure~\ref{fig:sphere_retraction} provides an illustration of the retraction from $ \mathrm{T}_{\f_{i}} \Stwo $ to $ \Stwo $. 

\begin{figure}[htbp]
  \centering
  \includegraphics[width=0.40\textwidth]{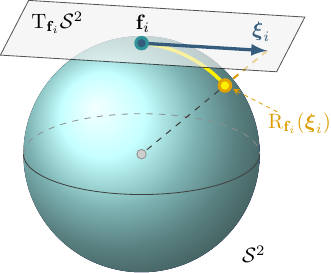}
  \caption{An illustration of the retraction mapping on the unit sphere $ \Stwo $.}\label{fig:sphere_retraction}
\end{figure}

For the power manifold $\Stwon$, the retraction of all the tangent vectors $ \boldxi_{i} $, $ i = 1, \dots, n $, is a mapping
\[
   \Retraction_{\f} \colon \mathrm{T}_{\f}\Stwon \to \Stwon,
\]
defined by
\begin{equation}\label{eq:retraction_Stwon}
   \begin{bmatrix} \boldxi_{1} & \cdots & \boldxi_{n} \end{bmatrix} \tr \mapsto \diag\!\left( \dfrac{1}{\|\f_{1} + \boldxi_{1}\|_{2}}, \ldots, \dfrac{1}{\|\f_{n} + \boldxi_{n}\|_{2}} \right) \begin{bmatrix} \f_{1}+\boldxi_{1} & \cdots & \f_{n}+\boldxi_{n} \end{bmatrix} \tr.
\end{equation}
Again, this retraction is implemented by row-wise normalization $\f + \boldxi$.

\section{Riemannian gradient descent method}\label{sec:Riem_grad_descent}

We are now in the position of introducing the RGD method. In this section, we will first provide the formula for the Riemannian gradient on the power manifold $\Stwon$. Then, we will explain the RGD method by providing its pseudocode, and finally, we will recall the known theoretical results that ensure the convergence of RGD.

The Riemannian gradient of the objective function $E$ in \eqref{eq:objective_function} is given by the projection onto $ \mathrm{T}_{\f} \Stwon $ of the Euclidean gradient
\begin{equation}\label{eq:rgrad_obj_fun}
    \grad E(f) = \PT(\nabla E(f)),
\end{equation}
where $\nabla E$ is explicitly formulated in \eqref{eq:egrad_obj_fun}.
This is always the case for embedded submanifolds.
Figure~\ref{fig:Riemannian_gradient} illustrates the difference between the Euclidean and the Riemannian gradient for one point on the unit sphere.

\begin{figure}[htbp]
\centering
\includegraphics[width=0.40\columnwidth]{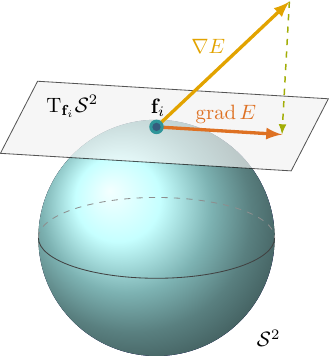}
\caption{Illustration of the difference between Euclidean and Riemannian gradient.}
\label{fig:Riemannian_gradient}
\end{figure}

Given an initial iterate $ f^{(0)} \in \Stwon $, the line-search algorithm generates a sequence of iterates $ \lbrace f^{(k)} \rbrace $ as follows.
At each iteration $ k = 0,1,2,\ldots $, it chooses a search direction $\boldd^{(k)} = -\grad E(f^{(k)})$ in the tangent space $\mathrm{T}_{\f^{(k)}}\Stwon$ such that the sequence $ \lbrace \boldd^{(k)} \rbrace $ is gradient related. Then the new point $ f^{(k+1)} $ is chosen such that
\begin{equation}\label{eq:condition_412}
   E(f^{(k)}) - E(f^{(k+1)}) \geq c \left( E(f^{(k)}) - E \! \left(\Retraction_{\f^{(k)}} (\alpha_{k}\boldd^{(k)})\right) \right),
\end{equation}
where $ \alpha_{k} $ is the step size for the given $ \boldd^{(k)} $.

The RGD method on the power manifold $\Stwon$ is summarized in Algorithm~\ref{algo:RGD}. It has been adapted from~\protect{\cite[p.~63]{AMS:2008}}. In practice, in our numerical experiments of section~\ref{sec:numerical_experiments}, the initial mapping $\f^{(0)}\in \Stwon$ is computed by applying a few steps of the fixed-point iteration (FPI) method; see section~\ref{sec:fixed_point_method}, Algorithm~\ref{algo:FPI_SEM}. The line-search procedure used in line 7 is described in detail in appendix~\ref{sec:line_search}.

\begin{algorithm}
\SetAlgoLined
 Given objective function $E$, power manifold $\Stwon$, initial iterate $\f^{(0)}\in \Stwon$, retraction $\Retraction_{\f}$ from $\mathrm{T}_{\f}\Stwon$ to $\Stwon$\;
 \KwResult{Sequence of iterates $\lbrace f^{(k)} \rbrace$.}
 $k \leftarrow 0$\;
 \While{$f^{(k)}$ sufficiently minimizes $E$}{
     Compute the Euclidean gradient of the objective function $\nabla E(f^{(k)})$ \eqref{eq:egrad_obj_fun}\;
     Compute the Riemannian gradient as $ \grad E(f^{(k)}) = \P_{\mathrm{T}_{\f^{(k)}}\Stwon} \! \big( \nabla E(f^{(k)}) \big)$\;
     Choose the anti-gradient direction $ \boldd^{(k)} = -\grad E(f^{(k)}) $\;
     Use a line-search procedure to compute a step size $\alpha_{k} > 0$ that satisfies the sufficient decrease condition; see appendix~\ref{sec:line_search}\;
     Set $\f^{(k+1)} = \Retraction_{\f^{(k)}}(\alpha_{k} \boldd^{(k)})$\;
     $ k \leftarrow k + 1 $\;
 }
 \caption{RGD on $\Stwon$.}\label{algo:RGD}
\end{algorithm}

\subsection{Known convergence results}\label{sec:convergence_results}

The RGD method has theoretically guaranteed convergence. For the reader's convenience, we report the two main results on the convergence of RGD. The first result is about the convergence of Algorithm~\ref{algo:RGD} to critical points of the objective function. The second result is about the convergence of RGD to a local minimizer with a line-search technique.

\begin{theorem}[\protect{\cite[Theorem~4.3.1]{AMS:2008}}]\label{thm:convg_of_LS}
Let $ \lbrace f^{(k)} \rbrace $ be an infinite sequence of iterates generated by Algorithm~\ref{algo:RGD}. Then, every accumulation point of $ \lbrace f_{k} \rbrace $ is a critical point of the cost function $ E $.
\end{theorem}

\begin{remark}
We are implicitly saying that a sequence can have more than one accumulation point; for example, from a sequence $ \lbrace x_{k} \rbrace $, we may extract two subsequences such that they have two distinct accumulation points.
\end{remark}

The proof of Theorem \ref{thm:convg_of_LS} can be done by contradiction, but it remains pretty technical, so we refer the interested reader to \protect{\cite[p.~65]{AMS:2008}}. It should be pointed out that Theorem \ref{thm:convg_of_LS} only guarantees the convergence to critical points. It does not tell us anything about their nature, i.e., whether the critical points are local minimizers, local maximizers, or saddle points. However, if $\lambda_{H,\min} > 0$, then $f_{\ast}$ is a local minimizer of $E$. \protect{\cite[\S 4.5.2]{AMS:2008}} gives an asymptotic convergence bound for Algorithm~\ref{algo:RGD} under this assumption (i.e., that $\lambda_{H,\min} > 0$). Indeed, the following result uses the smallest and largest eigenvalues of the Hessian of the objective function at a critical point $f_{\ast}$.

\begin{theorem}[\protect{\cite[Theorem~4.5.6]{AMS:2008}}]
    Let $ \lbrace f_{k} \rbrace $ be an infinite sequence of iterates generated by Algorithm~\ref{algo:RGD}, converging to a point $f_{\ast}$. (By Theorem~\ref{thm:convg_of_LS}, $f_{\ast}$ is a critical point of $E$.) Let $\lambda_{H,\min}$ and $\lambda_{H,\max}$ be the smallest and largest eigenvalues of the Hessian of $E$ at $f_{\ast}$. Assume that $\lambda_{H,\min} > 0$ (hence $f_{\ast}$ is a local minimizer of $E$). Then, given $r$ in the interval $(r_{\ast},1)$ with $r_{\ast} = 1-\min\left( 2\sigma \bar{\alpha}\lambda_{H,\min}, 4\sigma(1-\sigma)\beta \frac{\lambda_{H,\min}}{\lambda_{H,\max}} \right)$, there exists an integer $K \geq 0$ such that
    \begin{equation}\label{eq:418}
       E(f_{k+1}) - E(f_{\ast}) \leq (r+(1-r)(1-c))\left( E(f_{k}) - E(f_{\ast}) \right),
    \end{equation}
    for all $k\geq K$, where $c$ is the parameter in Algorithm~\ref{algo:RGD}.
    Note that $0<r_{\ast}<1$ since $\beta, \sigma \in (0,1)$.
\end{theorem}

As noted in \protect{\cite[p.~71]{AMS:2008}}, in typical cases of Algorithm~\ref{algo:RGD}, the constant $c$ in the descent condition equals 1, hence \eqref{eq:418} reduces to $E(f_{k+1}) - E(f_{\ast}) \leq r\left( E(f_{k}) - E(f_{\ast}) \right)$.

\section{Numerical experiments}\label{sec:numerical_experiments}

In this section, we showcase the convergence behavior of the RGD method using twelve mesh models and two line-search techniques. We present numerical results in tables and provide convergence plots. Additionally, we introduce a correction for bijectivity in section~\ref{sec:bijectivity_correction}, which helps to unfold the folding triangles.
We then compare RGD to the FPI method of~\cite{Yueh:2019} in section~\ref{sec:fixed_point_method} and the adaptive area-preserving parameterization method of~\cite{CGK:2022} in section~\ref{sec:aapp}.
Then, in section~\ref{sec:numerical_stability} we show that the algorithm is robust to noise by starting the algorithm from an initial guess with small perturbations.
Finally, in section~\ref{sec:brain_registration}, we apply our algorithm to the concrete application of surface registration of two brain models.

We conducted our experiments on a laptop Lenovo ThinkPad T460s, with Windows 10 Pro and MATLAB R2021a installed, with Intel Core i7-6600 CPU, 20GB RAM, and Mesa Intel HD Graphics 520. The benchmark triangular mesh models used in our numerical experiments are shown in Figure~\ref{fig:All_mesh_models}, arranged per increasing number of vertices and faces, from the top left to bottom right.

\begin{figure}[htbp]
\centering
\resizebox{\textwidth}{!}{
\begin{tabular}{lcccc}
\hline
Model Name  & Right Hand & David Head & Cortical Surface & Bull \\ 
\# Faces    & 8,808 & 21,338 & 30,000 & 34,504 \\ 
\# Vertices & 4,406 & 10,671 & 15,002 & 17,254 \\[0.25cm] 
&\includegraphics[height=3cm]{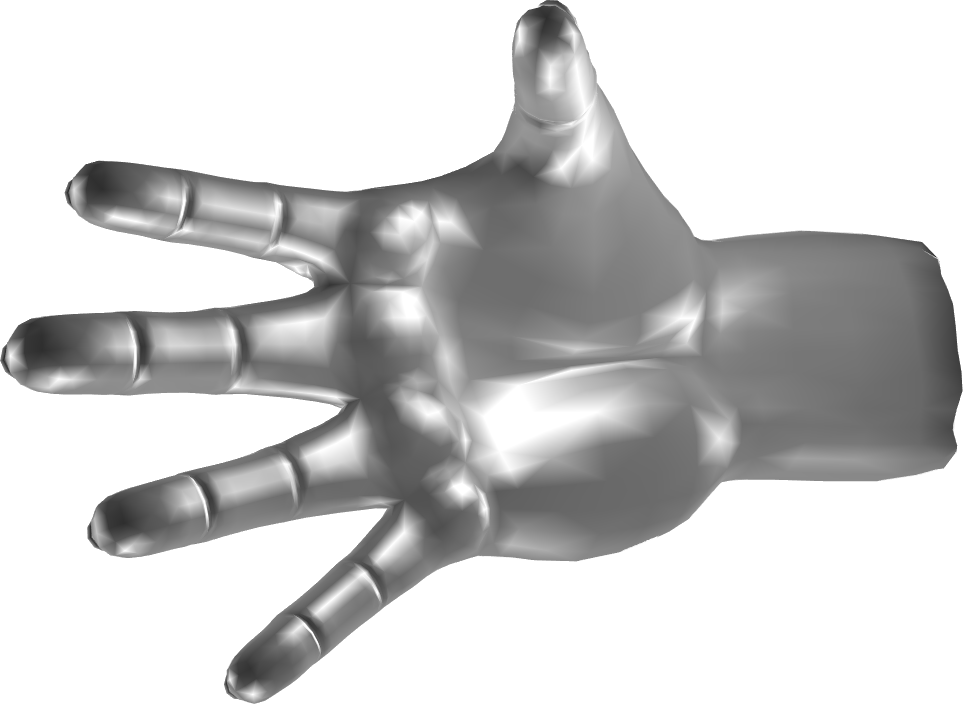} &
\includegraphics[height=3cm]{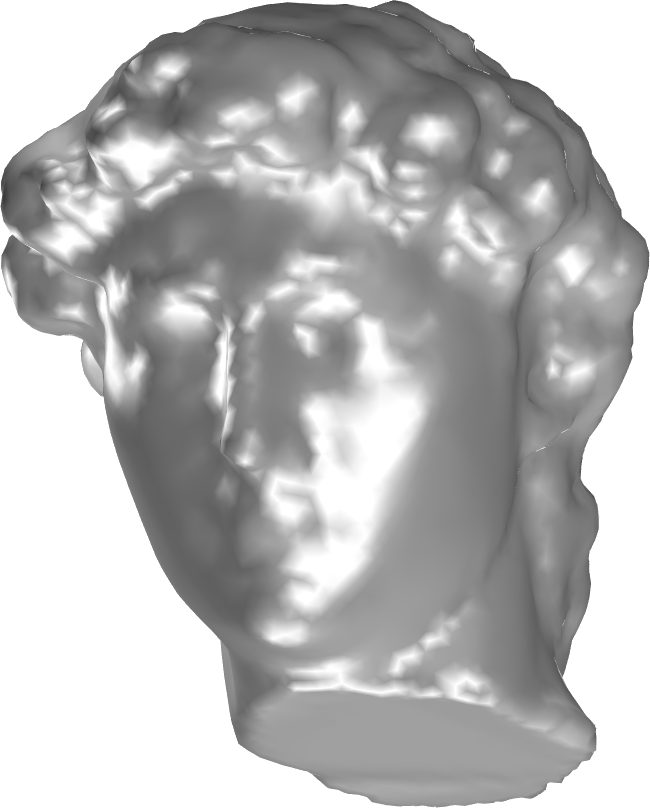} &
\includegraphics[height=3cm]{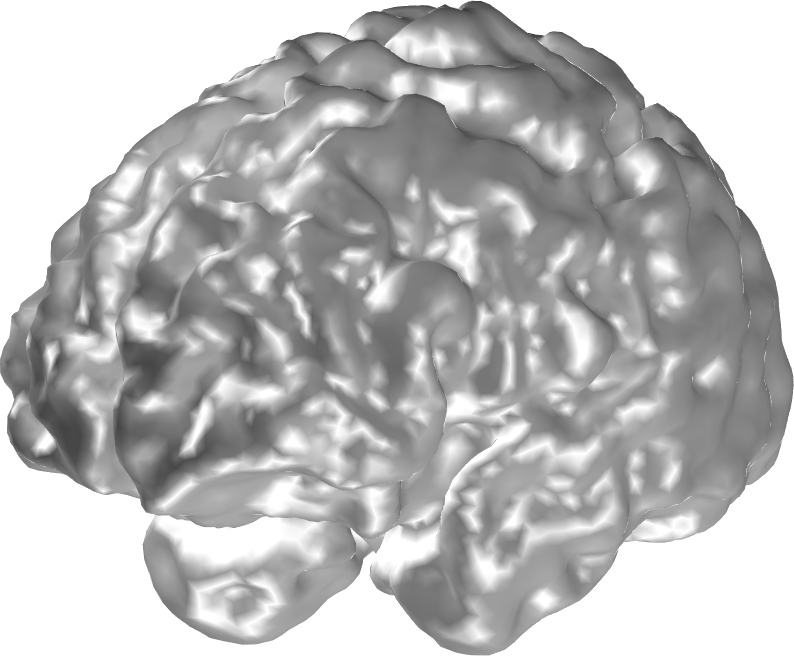} & 
\includegraphics[height=3cm]{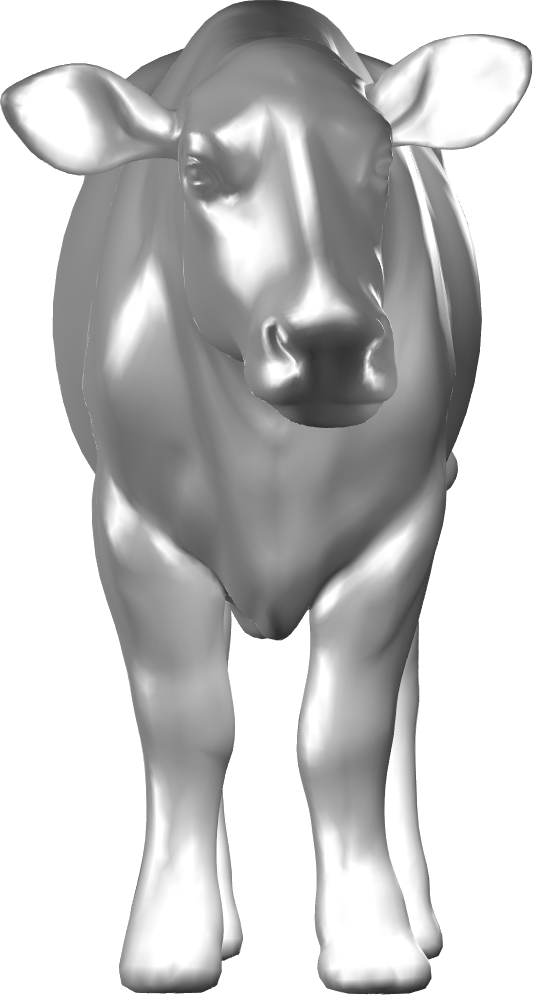} \\ \hline
\\\hline
Model Name  & Bulldog & Lion Statue & Gargoyle & Max Planck \\ 
\# Faces    & 99,590 & 100,000 & 100,000 & 102,212 \\ 
\# Vertices & 49,797 &  50,002 &  50,002 &  51,108 \\[0.25cm]  
&\includegraphics[height=3cm]{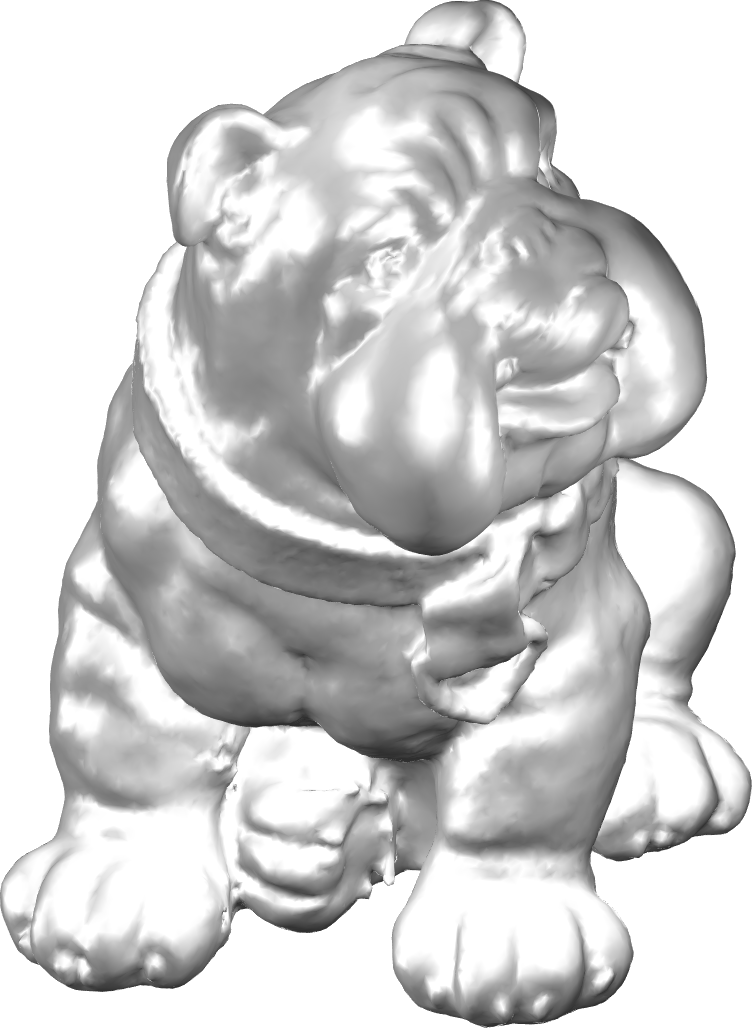} &
\includegraphics[height=3cm]{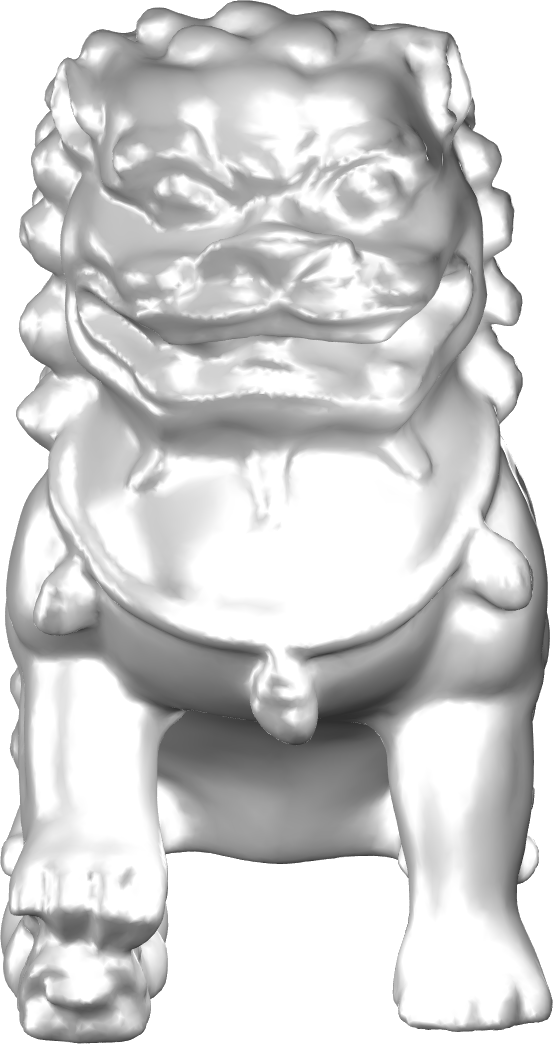} &
\includegraphics[height=3cm]{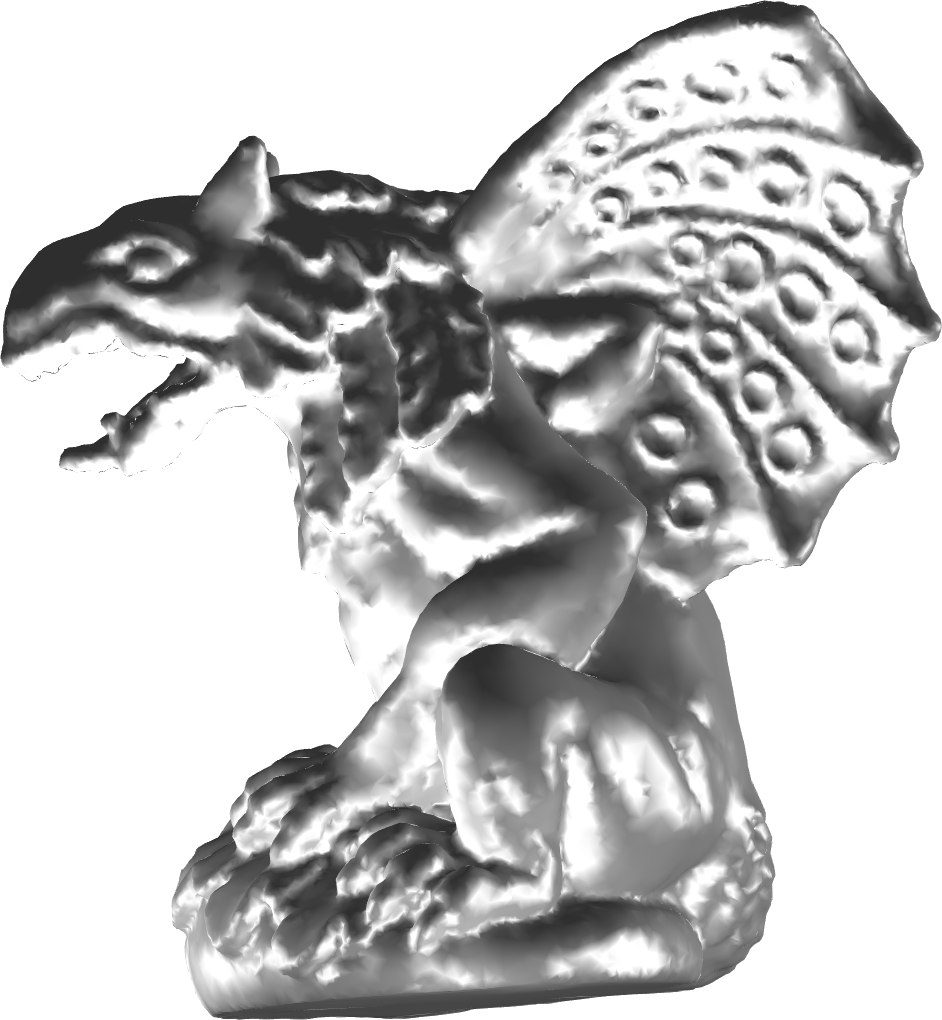} &
\includegraphics[height=3cm]{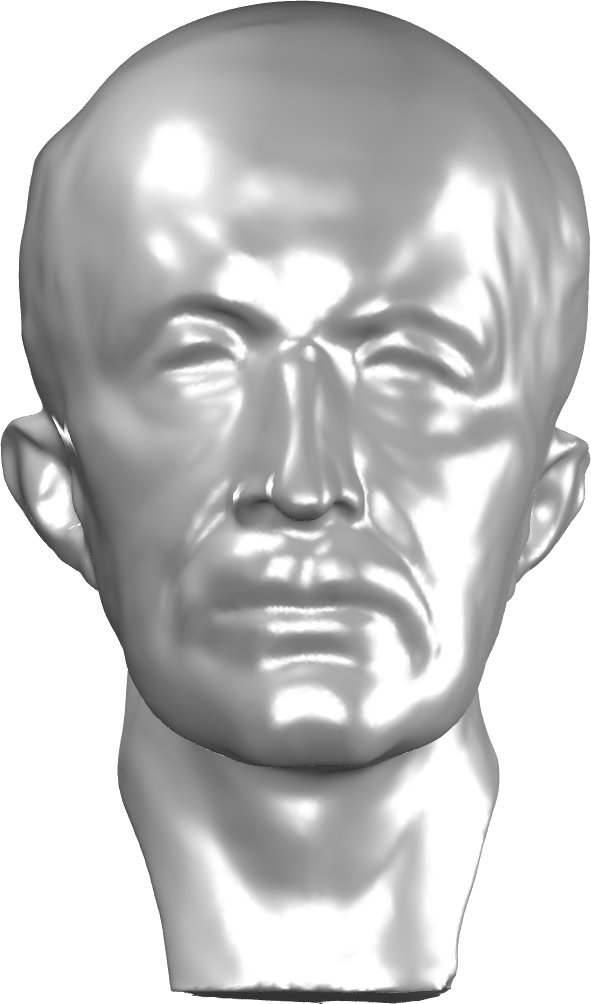} \\
\\\hline
Model Name  & Bunny & Chess King & Art Statuette & Bimba \\ 
\# Faces    & 111,364 & 263,712 & 895,274 & 1,005,146 \\ 
\# Vertices &  55,684 & 131,858 & 447,639 &   502,575 \\[0.25cm]  
&\includegraphics[height=3cm]{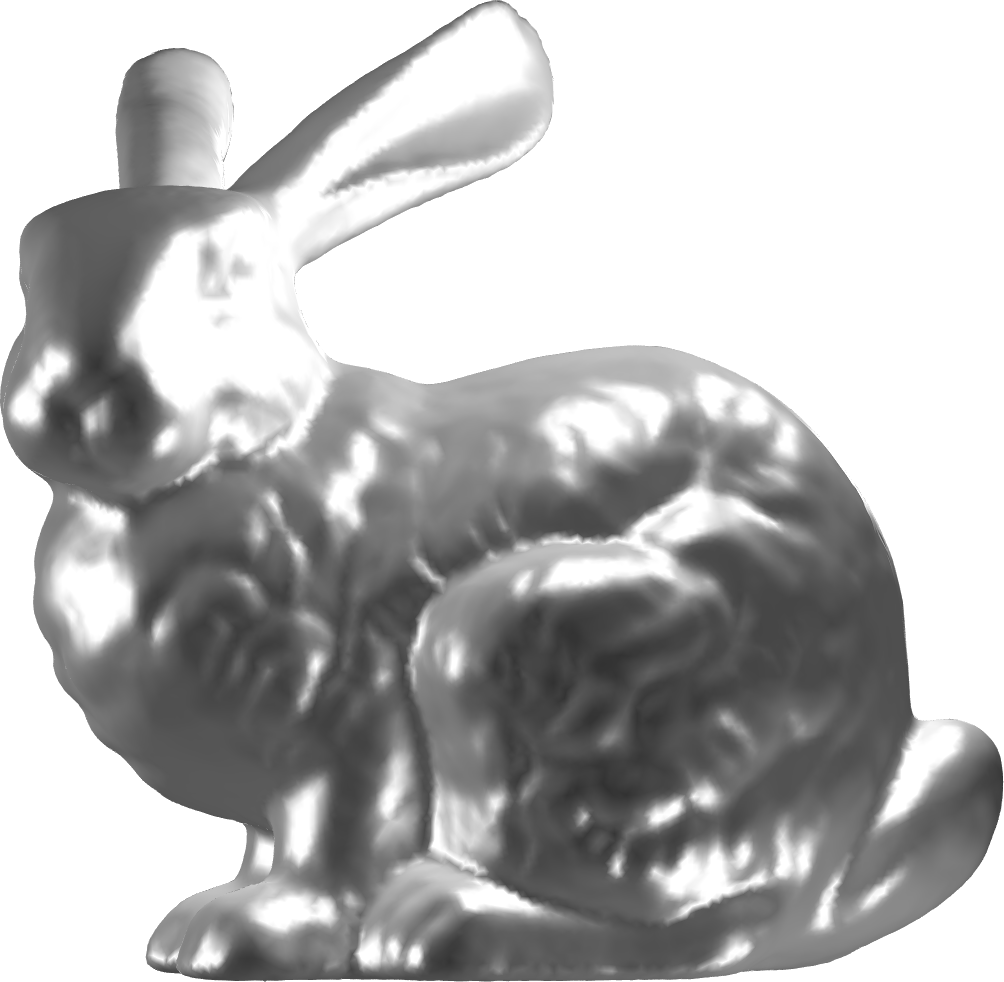} &
\includegraphics[height=3cm]{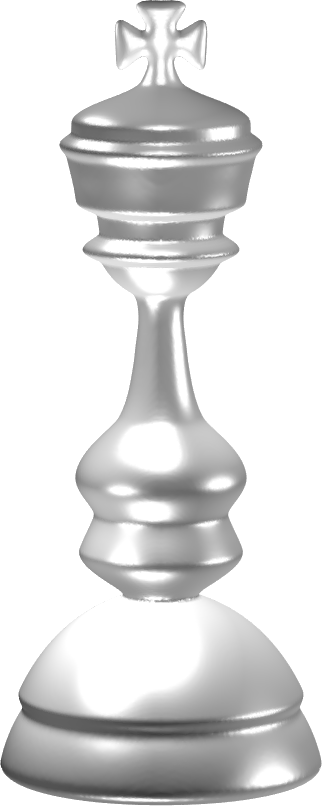} &
\includegraphics[height=3cm]{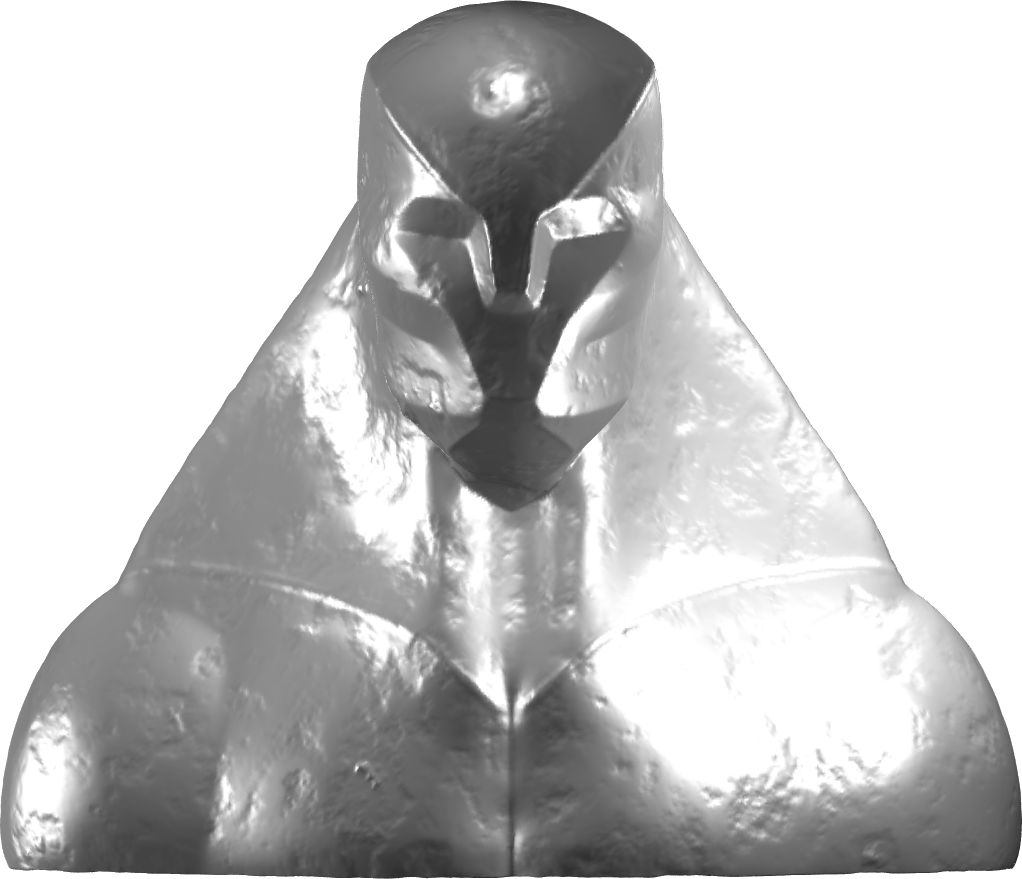} &
\includegraphics[height=3cm]{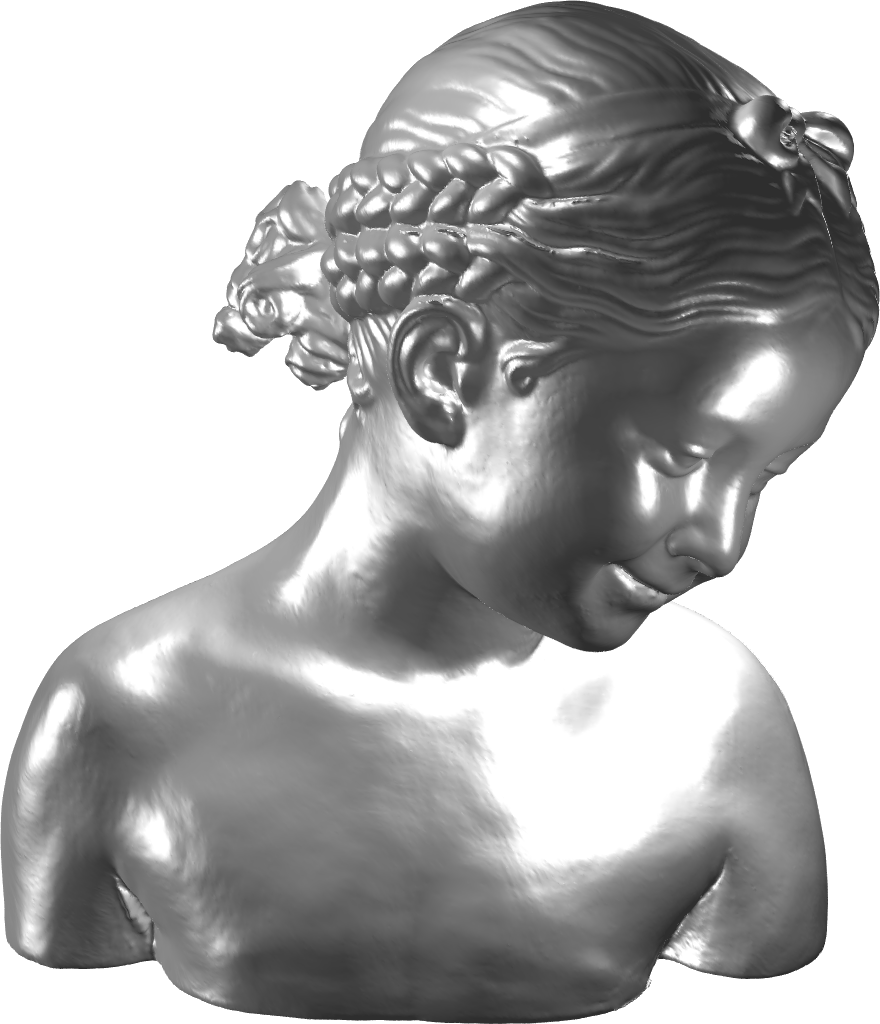} \\
\hline
\end{tabular}
}
\caption{The benchmark triangular mesh models used in this paper.}
\label{fig:All_mesh_models}
\end{figure}

\subsection{Convergence behavior of RGD}

To provide the RGD method with a good initial mapping, we first apply a few steps of the FPI method of~\cite{Yueh:2019}, described in section~\ref{sec:fixed_point_method}.
We adopted two different line-search strategies: one that uses MATLAB's \texttt{fminbnd} and another that uses the quadratic/cubic interpolant of \protect{\cite[\S 6.3.2]{DennisSchnabel:1996}}, described in appendix~\ref{sec:line_search}.

In all the experiments, we monitor the authalic energy $E_A(f) \coloneqq E_S(f) - \cA(f)$ instead of $E$ because when $E_{A} = 0$, we know from the theory that $f$ is an area-preserving mapping. Strictly speaking, in practice, we never obtain a mapping that exactly preserves the area, but we obtain a mapping that is area-distortion minimizing, since $E_{A}$ is never identically zero.
We also monitor the ratio between the standard deviation and the mean of the area ratio. This quantity has been considered in~\cite{CR:2018}. Finally, the computational time is always reported in seconds, and \#Fs denotes the number of folding triangles.

Tables~\ref{tab:RGD_ES_fminbnd} and~\ref{tab:RGD_ES_interpolant} report the numerical results for RGD for minimizing the (normalized) stretch energy $E_{S}$, when run for a maximum number of iterations (10 on the left and 100 on the right). Table~\ref{tab:RGD_ES_fminbnd} is for RGD which uses the \texttt{fminbnd} line-search strategy, while Table~\ref{tab:RGD_ES_interpolant} is for the RGD that uses the quadratic/cubic interpolant from \protect{\cite[\S 6.3.2]{DennisSchnabel:1996}}. 
Similarly to the tables, Figures~\ref{fig:All_models_NormalizedSE_LS_fminbnd_alphaMax_1_MaxIter_100} and \ref{fig:All_models_NormalizedSE_LS_interpolant_alphaMax_1_MaxIter_100} illustrate the convergence behavior of RGD in the same setting when run for 100 iterations.

\begin{table}[htbp]
   \caption{RGD for minimizing the (normalized) stretch energy $E_{S}$. Line-search strategy: \texttt{fminbnd}.}
   \label{tab:RGD_ES_fminbnd}
   \center
   \resizebox{\textwidth}{!}{
      \begin{tabular}{@{} *{9}{c}}
            \toprule
            & \multicolumn{4}{c}{$10$ Iterations} & \multicolumn{4}{c}{$100$ Iterations} \\
            \cmidrule(lr){2-5}  \cmidrule(lr){6-9}
            Model Name  &   SD/Mean  &  $E_{A}(f)$  &  Time  & \#Fs  &   SD/Mean  &  $E_{A}(f)$  &  Time  & \#Fs \\
            \cmidrule(r){1-1}  \cmidrule(lr){2-5}  \cmidrule(lr){6-9}
            {\footnotesize Right Hand}       &   0.1950   &  $ 2.17 \times 10^{-1}$  &    0.85    &   4   &   0.1459   &  $ 1.28 \times 10^{-1}$  &    9.37    &         2    \\
            {\footnotesize David Head}       &   0.0178   &  $ 4.04 \times 10^{-3}$  &    1.66    &   0   &   0.0156   &  $ 3.05 \times 10^{-3}$  &   14.95    &        0    \\
            {\footnotesize Cortical Surface} &   0.0214   &  $ 4.68 \times 10^{-3}$  &    3.62    &   0   &   0.0200   &  $ 4.15 \times 10^{-3}$  &   32.40    &        0    \\
            {\footnotesize Bull}             &   0.1492   &  $ 2.58 \times 10^{-1}$  &    4.99    &   4   &   0.1380   &  $ 2.31 \times 10^{-1}$  &   43.27    &        1    \\
            {\footnotesize Bulldog}          &   0.0369   &  $ 1.41 \times 10^{-2}$  &   13.59    &   0   &   0.0343   &  $ 1.27 \times 10^{-2}$  &  136.63    &       0    \\
            {\footnotesize Lion Statue}      &   0.1935   &  $ 4.77 \times 10^{-1}$  &   16.83    &   0   &   0.1922   &  $ 4.69 \times 10^{-1}$  &  160.20    &       0    \\
            {\footnotesize Gargoyle}         &   0.0690   &  $ 5.28 \times 10^{-2}$  &   19.55    &   0   &   0.0653   &  $ 4.84 \times 10^{-2}$  &  192.62    &       0    \\
            {\footnotesize Max Planck}       &   0.0537   &  $ 3.54 \times 10^{-2}$  &   16.52    &   0   &   0.0525   &  $ 3.38 \times 10^{-2}$  &  161.01    &       0    \\
            {\footnotesize Bunny}            &   0.0417   &  $ 2.18 \times 10^{-2}$  &   20.59    &   0   &   0.0404   &  $ 2.04 \times 10^{-2}$  &  226.99    &       0    \\
            {\footnotesize Chess King}       &   0.0687   &  $ 6.07 \times 10^{-2}$  &   52.36    &  21   &   0.0639   &  $ 5.14 \times 10^{-2}$  &  518.18    &      17    \\
            {\footnotesize Art Statuette}    &   0.0408   &  $ 2.14 \times 10^{-2}$  &   140.59   &   0   &   0.0405   &  $ 2.10 \times 10^{-2}$  &  1\,111.67   &       0    \\
            {\footnotesize Bimba Statue}     &   0.0514   &  $ 3.31 \times 10^{-2}$  &   270.63   &   1   &   0.0511   &  $ 3.29 \times 10^{-2}$  &  2\,320.19   &       1    \\
            \bottomrule
      \end{tabular}
      }
\end{table}

\begin{table}[htbp]
   \caption{RGD for minimizing the (normalized) stretch energy $E_{S}$. Line-search strategy: quadratic/cubic approximation from \protect{\cite[\S 6.3.2]{DennisSchnabel:1996}}.} 
   \label{tab:RGD_ES_interpolant}
   \center
   \resizebox{\textwidth}{!}{
      \begin{tabular}{@{} *{9}{c}}
            \toprule
            & \multicolumn{4}{c}{$10$ Iterations} & \multicolumn{4}{c}{$100$ Iterations} \\
            \cmidrule(lr){2-5}  \cmidrule(lr){6-9}
            Model Name  &   SD/Mean  &  $E_{A}(f)$  &  Time  & \#Fs  &   SD/Mean  &  $E_{A}(f)$  &  Time  & \#Fs \\
            \cmidrule(r){1-1}  \cmidrule(lr){2-5}  \cmidrule(lr){6-9}
            {\footnotesize Right Hand}    &   0.1936   &  $ 2.16 \times 10^{-1}$  &   0.36   &   4   &   0.1204   &  $ 9.40 \times 10^{-2}$  &    4.07    &       1    \\
            {\footnotesize David Head}    &   0.0178   &  $ 4.04 \times 10^{-3}$  &   0.99   &   0   &   0.0156   &  $ 3.04 \times 10^{-3}$  &    9.16    &       0    \\
            {\footnotesize Cortical Surface} &   0.0216   &  $ 4.75 \times 10^{-3}$  &   1.40   &   0   &   0.0200   &  $ 3.72 \times 10^{-3}$  &   16.01    &       0    \\
            {\footnotesize Bull}          &   0.1492   &  $ 2.59 \times 10^{-1}$  &   1.77   &   4   &   0.1348   &  $ 2.19 \times 10^{-1}$  &   18.89    &       1    \\
            {\footnotesize Bulldog}       &   0.0369   &  $ 1.41 \times 10^{-2}$  &   6.60   &   0   &   0.0343   &  $ 1.27 \times 10^{-2}$  &   61.93    &       0    \\
            {\footnotesize Lion Statue}   &   0.1935   &  $ 4.77 \times 10^{-1}$  &   7.75   &   0   &   0.1894   &  $ 4.54 \times 10^{-1}$  &   76.76    &       0    \\
            {\footnotesize Gargoyle}      &   0.0688   &  $ 5.26 \times 10^{-2}$  &   7.81   &   0   &   0.0646   &  $ 4.76 \times 10^{-2}$  &   80.52    &       0    \\
            {\footnotesize Max Planck}    &   0.0537   &  $ 3.54 \times 10^{-2}$  &   7.18   &   0   &   0.0525   &  $ 3.39 \times 10^{-2}$  &   75.60    &       0    \\
            {\footnotesize Bunny}         &   0.0417   &  $ 2.18 \times 10^{-2}$  &   8.30   &   0   &   0.0390   &  $ 1.91 \times 10^{-2}$  &   89.62   &       0    \\
            {\footnotesize Chess King}    &   0.0692   &  $ 6.07 \times 10^{-2}$  &  20.06   &  21   &   0.0647   &  $ 5.23 \times 10^{-2}$  &  207.47    &      17    \\
            {\footnotesize Art Statuette} &   0.0408   &  $ 2.14 \times 10^{-2}$  &  57.71   &   0   &   0.0405   &  $ 2.10 \times 10^{-2}$  &  654.57    &       0    \\
            {\footnotesize Bimba Statue}  &   0.0514   &  $ 3.31 \times 10^{-2}$  &  70.83   &   1   &   0.0512   &  $ 3.29 \times 10^{-2}$  &  775.36    &       1    \\
            \bottomrule            
      \end{tabular}
      }
\end{table}

\begin{figure}[htbp]
  \centering
  \includegraphics[width=\textwidth]{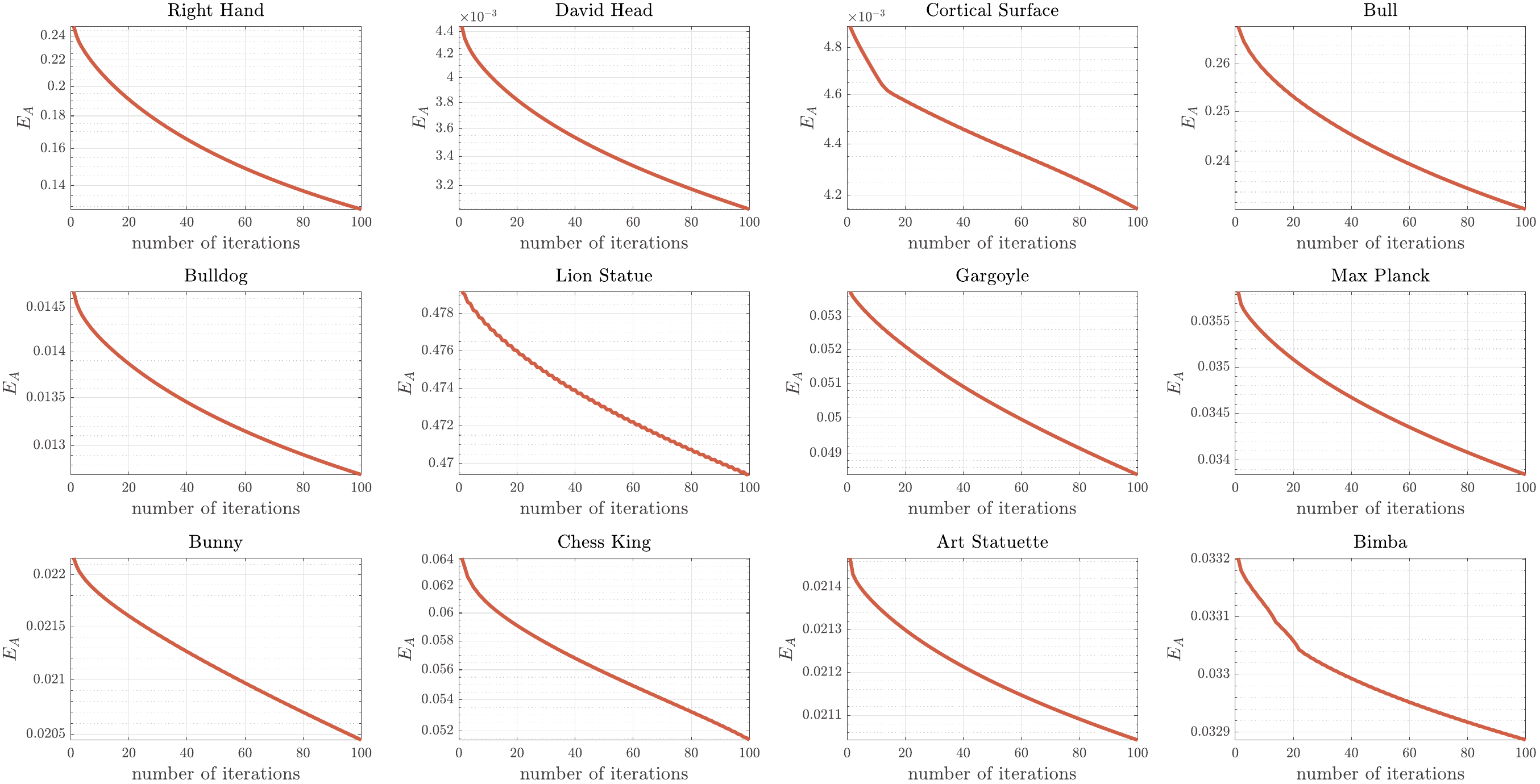}
  \caption{Convergence of the authalic energy for the benchmark mesh models, with the RGD method for minimizing the (normalized) stretch energy $E_{S}$. Line-search strategy: \texttt{fminbnd}.}
   \label{fig:All_models_NormalizedSE_LS_fminbnd_alphaMax_1_MaxIter_100}
\end{figure}

\begin{figure}[htbp]
  \centering
  \includegraphics[width=\textwidth]{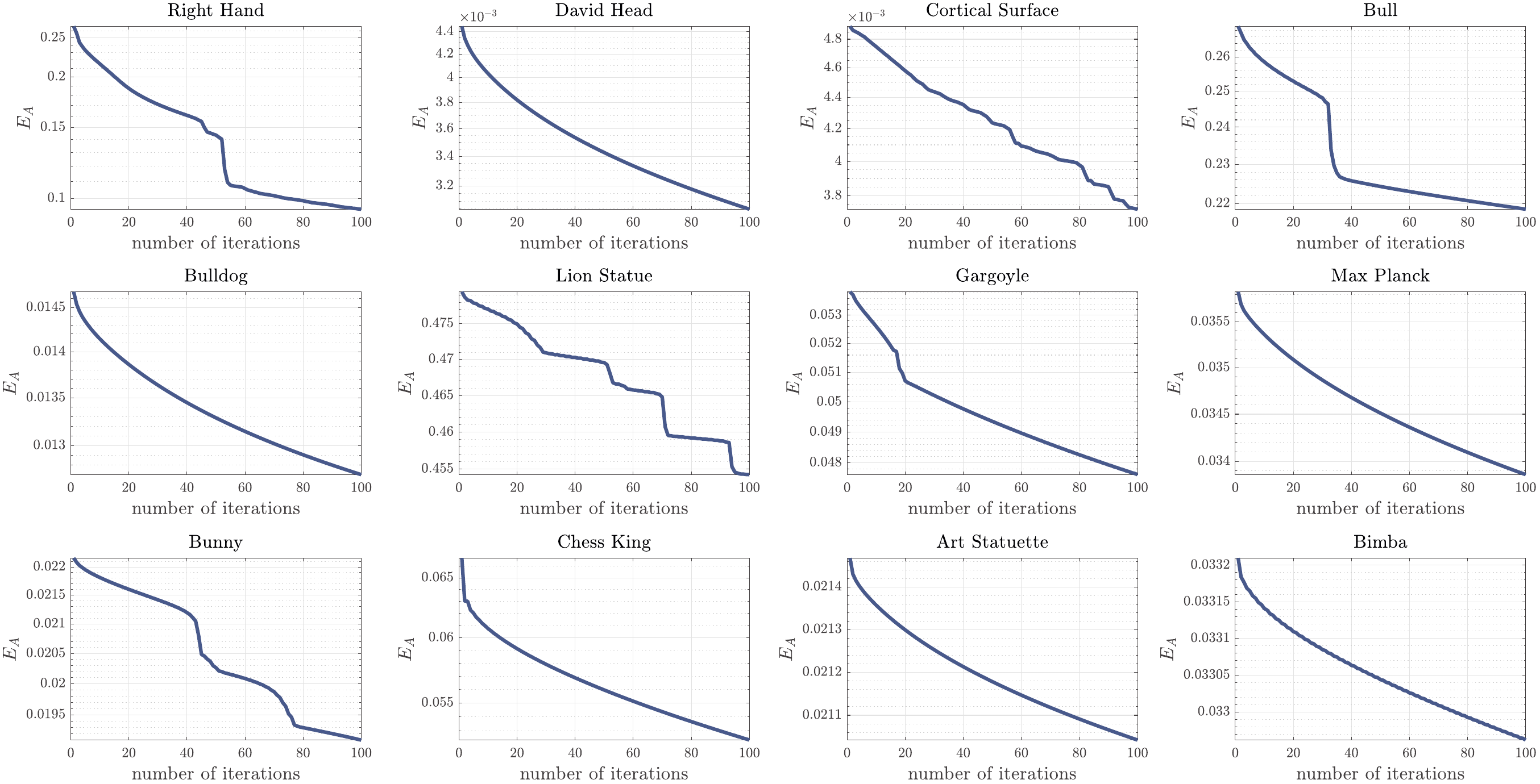}
  \caption{Convergence of the authalic energy for the benchmark mesh models, with the RGD method for minimizing the (normalized) stretch energy $E_{S}$. Line-search strategy: quadratic/cubic approximation from \protect{\cite[\S 6.3.2]{DennisSchnabel:1996}}.}
   \label{fig:All_models_NormalizedSE_LS_interpolant_alphaMax_1_MaxIter_100}
\end{figure}

A comparison of the computational times between the two line-search strategies shows that the line-search strategy with a quadratic/cubic interpolant (Table~\ref{tab:RGD_ES_interpolant}) is much more efficient than the line-search strategy that uses MATLAB's \texttt{fminbnd} (Table~\ref{tab:RGD_ES_fminbnd}). In many cases, the former even yields more accurate results than the latter. This is particularly evident for the Art Statuette (1\,111.67 s versus 654.57 s) and the Bimba Statue (2\,320.19 s versus 775.36 s) mesh models.

The numerical results show that, in general, RGD can still decrease energy as the number of iterations increases. 

Figure~\ref{fig:All_models_NormalizedSE_LS_interpolant_alphaMax_1_MaxIter_10000} shows the convergence behavior for the first three smallest mesh models considered, namely Right Hand, David Head, and Cortical Surface, when the algorithm is run for many more iterations, here 10\,000 iterations. It shows that RGD keeps decreasing the authalic energy $E_{A}$, albeit very slowly so. The values of SD/Mean are also improved for all three mesh models. The corresponding numerical results are reported in Table~\ref{tab:RGD_ES_interpolant_10000}; results for 100 iterations are also reported for easier comparison.

\begin{table}[htbp]
   \caption{RGD for minimizing the (normalized) stretch energy $E_{S}$. Line-search strategy: quadratic/cubic approximation from \protect{\cite[\S 6.3.2]{DennisSchnabel:1996}}. Results for 10\,000 iterations for the three smallest mesh models considered.} 
   \label{tab:RGD_ES_interpolant_10000}
   \center
   \resizebox{\textwidth}{!}{
      \begin{tabular}{@{} *{9}{c}}
            \toprule
            & \multicolumn{4}{c}{$100$ Iterations} & \multicolumn{4}{c}{$10\,000$ Iterations} \\
            \cmidrule(lr){2-5}  \cmidrule(lr){6-9}
            Model Name &   SD/Mean  &  $E_{A}(f)$  &  Time  & \#Fs  &   SD/Mean  &  $E_{A}(f)$  &  Time  & \#Fs \\
            \cmidrule(r){1-1}  \cmidrule(lr){2-5}  \cmidrule(lr){6-9}
            {\footnotesize Right Hand}    &   0.1204   &  $ 9.40 \times 10^{-2}$  &    4.07    &       1    &   0.0545   &  $ 2.07 \times 10^{-2}$  &    431.28    &       0    \\
            {\footnotesize David Head}    &   0.0156   &  $ 3.04 \times 10^{-3}$  &    9.16    &       0    &   0.0029   &  $ 1.01 \times 10^{-4}$  &    1\,018.87    &       0    \\
            {\footnotesize Cortical Surface} &   0.0200   &  $ 3.72 \times 10^{-3}$  &   16.01    &       0    &   0.0045   &  $ 2.06 \times 10^{-4}$  &    1\,328.56    &       0    \\
            \bottomrule
      \end{tabular}
   }
\end{table}

\begin{figure}[htbp]
  \centering
  \includegraphics[width=\textwidth]{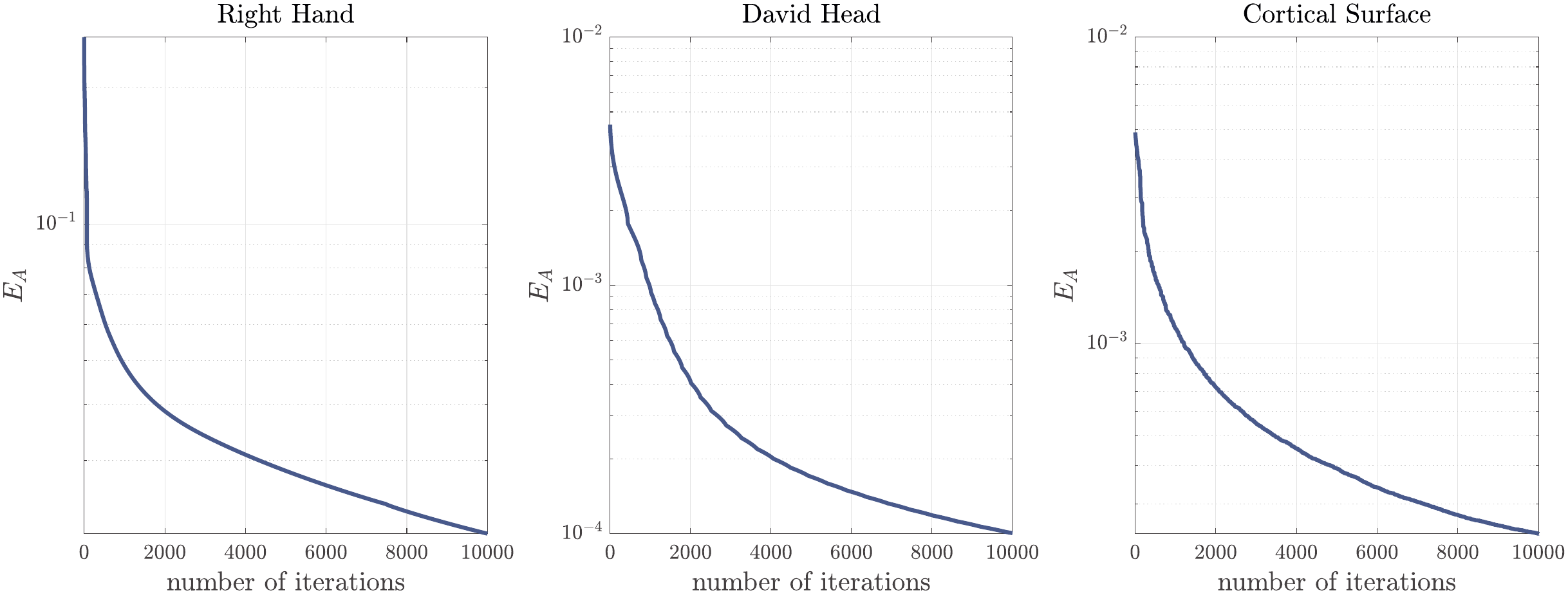}
  \caption{Convergence of the authalic energy for the three smallest benchmark mesh models, with the RGD method for minimizing the (normalized) stretch energy $E_{S}$. Line-search strategy: quadratic/cubic approximation from \protect{\cite[\S 6.3.2]{DennisSchnabel:1996}}. 10\,000 iterations.}
   \label{fig:All_models_NormalizedSE_LS_interpolant_alphaMax_1_MaxIter_10000}
\end{figure}

We compute with MATLAB the eigenvalues of the Hessian matrix at the minimizer.
Table~\ref{tab:eig_hessian} reports on the smallest eigenvalue of the Hessian of the initial mapping (produced by the first few iterations of the FPI method) and eigenvalues of the Hessian of the mapping when the stopping criterion of RGD is achieved.
The eigenvalues of the Hessian are computed by the MATLAB built-in function \texttt{eigs} with the option \texttt{smallestabs} and the number of eigenvalues to compute being 2\,000. 
We observe that the eigenvalues of the Hessian are significantly closer to zero after running the RGD method, as we expected.

\begin{table}[htbp]
   \caption{The smallest eigenvalue of the Hessian of the initial mapping given by FPI and after 10\,000 iterations of RGD.} 
   \label{tab:eig_hessian}
   \begin{center}
      \begin{tabular}{@{} *{3}{c}}
      \toprule
                          &   \multicolumn{2}{c}{Smallest Eigenvalue of the Hessian}  \\
            \cmidrule(lr){2-3}
            Model Name    &   After FPI  &  After RGD \\
            \hline
            Right Hand         &  $ -2.4429 \times 10^{0} $  &  $ -1.9013 \times 10^{-1} $  \\
            David Head         &  $  -1.8481 \times 10^{0} $  &  $  -2.8445 \times 10^{-6} $  \\
            Cortical Surface   &  $  -4.4940 \times 10^{-1} $  &  $  -3.6768 \times 10^{-3} $  \\
            \bottomrule
      \end{tabular}
   \end{center}
\end{table}

\subsection{Correction for bijectivity}\label{sec:bijectivity_correction}

The proposed RGD method does not guarantee the produced mapping is bijective. To remedy this drawback, we introduce a post-processing method to ensure bijectivity in the mapping. This is achieved by employing a modified version of the mean value coordinates, as described in~\protect{\cite{Floater:2003}}. 

Suppose we have a spherical mapping $f\colon \cM\to\Stwo$ that may not be bijective. First, we map the spherical image to the extended complex plane $\overline{\mathbb{C}} = \mathbb{C}\cup\{\infty\}$ by the stereographic projection 
\begin{equation} \label{eq:Ste}
\Pi_{\cS^2}(x,y,z) = \frac{x}{1-z} + \mathrm{i} \frac{y}{1-z}.
\end{equation}
Denote the mapping $h = \Pi_{\Stwo}\circ f$ and the complex-valued vector $\mathbf{h}_i = h(v_i)$, for $v_i\in\mathcal{V}(\cM)$. 
The folding triangular faces in the southern hemisphere are now mapped in $\mathbb{D}\subset\overline{\mathbb{C}}$, which can be unfolded by solving the linear systems
\begin{equation} \label{eq:meanValueLS}
[L_M(h)]_{\mathtt{I},\mathtt{I}} \widetilde{\mathbf{h}_{\mathtt{I}}} = -[L_M(h)]_{\mathtt{I},\mathtt{B}} \mathbf{h}_\mathtt{B},
\end{equation}
where $\mathtt{I} = \{ i \mid | h(v_i) | < r \}$ denotes the vertex index set with $r$ being a value slightly larger than $1$, e.g., $r=1.2$, $\mathtt{B} = \{1, \ldots, n\} \backslash \mathtt{I}$, and $L_{M}$ is the Laplacian matrix defined as
\begin{subequations} \label{eq:L_M}
\begin{equation} 
{[L_M(h)]}_{i,j} =
   \begin{cases}
   -\sum_{[v_i,v_j,v_k]\in\mathcal{F}(\cM)} [\omega_M(h)]_{i,j,k}  &\mbox{if $[{v}_i,{v}_j]\in\mathcal{E}(\cM)$,}\\
   -\sum_{\ell\neq i} [L_M(h)]_{i,\ell} &\mbox{if $j = i$,}\\
   0 &\mbox{otherwise,}
   \end{cases}
\end{equation}
with $\omega_M(h)$ being a variant of the mean value weight \cite{Floater:2003} defined as
\begin{equation}\label{eq:MVW}
[\omega_M(h)]_{i,j,k} = \frac{1}{\|\mathbf{h}_i - \mathbf{h}_j\|} \tan\frac{\varphi_{i,j}^k(h)}{2}
\end{equation}
\end{subequations}
in which $\varphi_{i,j}^k(h)$ is the angle opposite to the edge $[h(v_j), h(v_k)]$ at the point $h(v_i)$ on $h(\cM)$, as illustrated in Figure \ref{fig:MVW}. 
Then, the mapping is updated by replacing $\mathbf{h}_{\mathtt{I}}$ with $\widetilde{\mathbf{h}_{\mathtt{I}}}$. 
Next, an inversion 
\begin{equation} \label{eq:Inv}
\mathrm{Inv}(z) = \frac{1}{\bar{z}}
\end{equation}
is performed to reverse the positions of the southern and northern hemispheres. Then, the linear system \eqref{eq:meanValueLS} is solved again to unfold the triangular faces originally located in the northern hemisphere. We denote the updated mapping as $\widetilde{h}$. Ultimately, the corrected mapping is given by $\Pi_{\cS^2}^{-1}\circ \widetilde{h}$, where $\Pi_{\cS^2}^{-1}$ denotes the inverse stereographic projection
\begin{equation} \label{eq:invSte}
\Pi_{\cS^2}^{-1}(u+\mathrm{i}v) = \left( \frac{2 u}{u^2+v^2+1}, \ \frac{2 v}{u^2+v^2+1}, \ \frac{u^2+v^2-1}{u^2+v^2+1} \right).
\end{equation}

\begin{figure}
\center
\begin{tikzpicture}[thick,scale=1.2]
\coordinate (v_i) at (0,0);
\coordinate (v_j) at (0,2);
\coordinate (v_k) at (2,1);
\coordinate (v_l) at (-2,1);
\filldraw[green!20] (v_i) -- (v_j) -- (v_k);
\filldraw[green!20] (v_i) -- (v_j) -- (v_l);
\pic[draw, <-, "$~~\varphi_{i,j}^\ell(f)$", angle eccentricity=2]{angle = v_k--v_i--v_j};
\pic[draw, ->, "$\varphi_{i,j}^k(f)~~$", angle eccentricity=2]{angle = v_j--v_i--v_l};
\draw{
(v_i) -- (v_j) -- (v_k) -- (v_i) -- (v_l) -- (v_j)
};
\tikzstyle{every node}=[circle, draw, fill=yellow!20,
                        inner sep=1pt, minimum width=2pt]
\draw{
(0,0) node{$\f_i$}
(0,2) node{$\f_j$}
(2,1) node{$\f_\ell$}
(-2,1) node{$\f_k$}
};
\end{tikzpicture}
\caption{An illustration for the mean value weight \cite{Floater:2003}.}
\label{fig:MVW}
\end{figure}
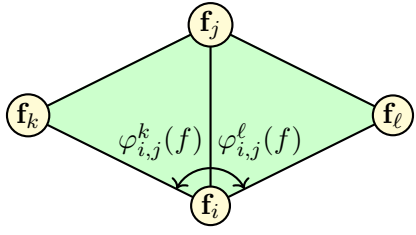

In practice, in our numerical experiments, we perform the bijectivity correction both after the FPI method and after the RGD method.

\subsection{Comparison with the FPI method}\label{sec:fixed_point_method}

To validate the algorithm, we compare it with the numerical results on the same models obtained using the FPI method. This method was proposed by Yueh et al.~\protect{\cite{Yueh:2019}} to calculate area-preserving boundary mapping while parameterizing 3-manifolds.

The FPI method of stretch energy minimization (SEM) is performed as follows. First, a spherical harmonic mapping $f^{(0)}\colon \cM\to\cS^2$ is computed by the Laplacian linearization method \cite{AHT:1999}. 
Suppose a spherical mapping $f^{(k)}$ is computed. 
Then, we map the spherical image to the extended complex plane by the stereographic projection \eqref{eq:Ste}
together with an inversion \eqref{eq:Inv}. 
Then, we define the mapping $h^{(k)}\colon \cM\to\mathbb{C}\cup\{\infty\}$
as
$$
h^{(k)}(v) = \mathrm{Inv} \circ\Pi_{\cS^2}\circ f^{(k)}(v).
$$
The mapping $h^{(k)}$ is represented as the complex-valued vector by $\mathbf{h}_i = h^{(k)}(v_i)$. 
Next, we define the interior vertex index set as
$$
\mathtt{I}^{(k)} = \{ i \mid | h^{(k)}(v_i) | < r \},
$$
which collects the indices of vertices that are located in the circle centered at the origin of radius $r$. 
Other vertices are defined as boundary vertices, and the associated index set is defined as 
$$
\mathtt{B}^{(k)} = \{1, \ldots, n\} \backslash \mathtt{I}^{(k)}.
$$
Then, the interior mapping is updated by solving the linear system
$$
[L_S(f^{(k)})]_{\mathtt{I}^{(k)},\mathtt{I}^{(k)}} \mathbf{h}^{(k+1)}_{\mathtt{I}^{(k)}} = -[L_S(f^{(k)})]_{\mathtt{I}^{(k)},\mathtt{B}^{(k)}} \mathbf{h}^{(k)}_{\mathtt{B}^{(k)}},
$$
while the boundary mapping remains the same, i.e., $\mathbf{h}^{(k)}_{\mathtt{B}^{(k+1)}} = \mathbf{h}^{(k)}_{\mathtt{B}^{(k)}}$.
The updated spherical mapping $f^{(k+1)}\colon\cM\to\cS^2$ is computed by
$$
f^{(k+1)}(v_i) = \Pi_{\cS^2}^{-1}(\mathbf{h}^{(k)}_i),
$$
where $\Pi_{\cS^2}^{-1}$ is the inverse stereographic projection \eqref{eq:invSte}.

\begin{algorithm}
\SetAlgoLined
Given a genus-zero closed mesh $\cM$, a tolerance $\varepsilon$, a radius $r$ (e.g., $\varepsilon=10^{-6}$, $r=1.2$)\;
 \KwResult{A spherical area-preserving parameterization $\mathbf{f}$.}
Compute a spherical conformal map $\mathbf{g}$ using the Laplacian linearization method \cite{AHT:1999}\;
Perform the stereographic projection  
$\mathbf{h}_\ell = \Pi_{\cS^2}(\mathbf{g}_\ell), \ell=1,\ldots,n$, as in \eqref{eq:Ste}\;
Let $\delta = \infty$\;
 \While{$\delta>\varepsilon$}{
     Update the matrix $L \leftarrow L_{S}({f})$, where $L_{S}(f)$ is defined as in \eqref{eq:L_S}\;
     Perform the inversion $\mathbf{h}_\ell \gets \mathrm{Inv}(\mathbf{h}_\ell)$, $\ell = 1, \ldots, n$, as in \eqref{eq:Inv};
     Update the index sets $\mathtt{I}=\left\{i \,\left|\, |\mathbf{h}_i| < r \right.\right\} \text{ and } \mathtt{B} = \{1, \ldots, n\} \backslash \mathtt{I}$\;
     Update $\mathbf{h}$ by solving the linear system $L_{\mathtt{I},\mathtt{I}} \mathbf{h}_{\mathtt{I}} = - L_{\mathtt{I},\mathtt{B}} \mathbf{h}_{\mathtt{B}}$\;
     Update $\mathbf{h}\gets \mathbf{h} / \mathrm{median}_{\ell}\|\mathbf{h}_\ell\|$\;
     Let $\mathbf{f}_\ell\leftarrow\Pi_{\cS^2}^{-1}(\mathbf{h}_\ell)$ as in \eqref{eq:invSte}, $\ell=1, \ldots, n$\;
     Update $\delta \leftarrow {E}_{S}(\mathbf{g}) - {E}_{S}(\mathbf{f})$\;
     Update $\mathbf{g}\leftarrow \mathbf{f}$.
 }
 \caption{FPI method of the SEM \cite{Yueh:2019}.}\label{algo:FPI_SEM}
\end{algorithm}

\begin{table}[htbp]
   \caption{FPI method for minimizing the authalic energy $E_{A}$, using two different stopping criteria. \#Its denotes the number of iterations at which the energy started to increase.}
   \label{tab:FPI}
   \center
   \resizebox{\textwidth}{!}{
      \begin{tabular}{@{} *{10}{c}}
            \toprule
            & \multicolumn{5}{c}{Energy $E_{A}$ Increased} & \multicolumn{4}{c}{$100$ Iterations} \\
            \cmidrule(lr){2-6}  \cmidrule(lr){7-10}
            Model Name  &   SD/Mean  &  $E_{A}(f)$  &  Time &  {\footnotesize \#Fs}  & \#Its &   SD/Mean  &  $E_{A}(f)$  &  Time  & {\footnotesize \#Fs} \\
            \cmidrule(r){1-1}  \cmidrule(lr){2-6}  \cmidrule(lr){7-10}
            {\footnotesize Right Hand}       &   0.2050   &  $ 2.86 \times 10^{-1}$  &   0.08   & 12  &   6  &   0.4598   &  $ 2.92 \times 10^{0}$  &    1.35    &      67    \\
            {\footnotesize David Head}       &   0.0191   &  $ 4.66 \times 10^{-3}$  &   0.35   &  0  &   8  &   0.0169   &  $ 3.58 \times 10^{-3}$  &    4.30    &       0    \\
            {\footnotesize Cortical Surface} &   0.0220   &  $ 4.93 \times 10^{-3}$  &   0.85   &  0  &  15  &   0.0174   &  $ 3.21 \times 10^{-3}$  &    5.62    &       0    \\
            {\footnotesize Bull}             &   0.1504   &  $ 2.74 \times 10^{-1}$  &   1.29   &  8  &  18  &   0.1876   &  $ 4.59 \times 10^{-1}$  &    6.90    &      40    \\
            {\footnotesize Bulldog}          &   0.0381   &  $ 1.49 \times 10^{-2}$  &   2.61   &  0  &  10  &   0.1833   &  $ 3.99 \times 10^{-1}$  &   22.22    &      53    \\
            {\footnotesize Lion Statue}      &   0.1940   &  $ 5.10 \times 10^{-1}$  &   1.12   &  1  &   4  &   0.2064   &  $ 5.28 \times 10^{-1}$  &   23.67    &      38    \\
            {\footnotesize Gargoyle}         &   0.0704   &  $ 5.47 \times 10^{-2}$  &   2.64   &  0  &  11  &   4.1020   &  $ 4.85 \times 10^{2}$   &   36.10    &     1955   \\
            {\footnotesize Max Planck}       &   0.0544   &  $ 3.67 \times 10^{-2}$  &   1.35   &  0  &   5  &   0.1844   &  $ 1.67 \times 10^{1}$   &   25.99    &     144    \\
            {\footnotesize Bunny}            &   0.0423   &  $ 2.24 \times 10^{-2}$  &   6.40   &  0  &  20  &   0.0394   &  $ 3.96 \times 10^{-2}$  &   35.78    &       2    \\
            {\footnotesize Chess King}       &   0.0713   &  $ 6.91 \times 10^{-2}$  &   6.35   &  9  &   8  &   1.0903   &  $ 1.79 \times 10^{1}$   &   88.04    &    1655    \\
            {\footnotesize Art Statuette}    &   0.0411   &  $ 2.15 \times 10^{-2}$  &  23.27   &  0  &   7  &   0.0908   &  $ 1.07 \times 10^{-1}$  &  342.95    &     126    \\
            {\footnotesize Bimba Statue}     &   0.0515   &  $ 3.32 \times 10^{-2}$  &  29.94   &  6  &   9  &   0.0932   &  $ 7.42 \times 10^{-2}$  &  305.00    &     144    \\
            \bottomrule
      \end{tabular}
      }
\end{table}

Table~\ref{tab:FPI} reports the numerical results for the FPI method applied to the twelve benchmark mesh models considered. Two different stopping criteria are considered: increase in authalic energy (columns 2 to 6) and maximum number of iterations (columns 7 to 10).
From Table~\ref{tab:FPI}, it appears that performing more iterations of the FPI method does not necessarily improve the mapping $f$. In most of the cases, except for David Head and Cortical Surface, the authalic energy and the values of SD/Mean increase. This motivates us to use the FPI method only to calculate a good initial mapping and then switch to RGD.

\subsection{Comparison with the adaptive area-preserving parameterization}\label{sec:aapp}

In this section, we compare the numerical results of our RGD method with the adaptive area-preserving parameterization for genus-0 closed surfaces proposed by Choi et al.~\protect{\cite{CGK:2022}}. 
The computational procedure is summarized as follows. 
First, the mesh is punctured by removing two triangular faces $\tau_1$ and $\tau_2$ that share a common edge. Then, the FPI of the SEM \cite{YLWY:2019} is applied to compute an area-preserving initial mapping ${g}_0\colon \cM\backslash\{\tau_1,\tau_2\}\to\mathbb{D} \coloneqq \{\mathbf{x}\in\mathbb{R}^2 \mid \|\mathbf{x}\|_2\leq 1\}$. The Beltrami coefficient of the mapping ${g}_0$ is denoted by $\mu_{g_0}$. Next, a quasi-conformal map $g\colon \cM\backslash\{\tau_1,\tau_2\}\to\mathbb{D}$ with $\|\mu_g\|_\infty = \|\lambda\mu_{g_0}\|_\infty <1$ is constructed \cite{CLL:2015}. The scaling factor $\lambda\in[0,1]$ is chosen to be $0.2$ in practice. Then, the optimal mass transport mapping $h_r\colon \cM\backslash\{\tau_1,\tau_2\}\to r\mathbb{D}$ is computed by the method proposed by Zhao et al. \cite{ZSGK:2013}, with the optimal radius $r$ satisfying
$$
r = \argmin_r\int |\mu_{(h_r\circ rg)^{-1}(z)}|^2 \mathrm{d}z.
$$
The final spherical area-preserving parameterization is obtained by the composition mapping $f=\Pi_{\Stwo}^{-1}\circ h_r \circ rg$, where $\Pi_{\Stwo}^{-1}$ denotes the inverse stereographic projection \eqref{eq:invSte}.

Table~\ref{tab:Choi_et_al} reports on the numerical results. The algorithm of Choi et al. \cite{CGK:2022} was run with the default parameters found in the code package\footnote{Available at \url{https://www.math.cuhk.edu.hk/~ptchoi/files/adaptiveareapreservingmap.zip}.}. The only satisfactory results are those for the David Head mesh model. However, even in this case, the values of the monitored quantities (SD/Mean, $E_{A}(f)$, and computation time) do not compete with those obtained by running our RGD method for 10 iterations; compare with the fourth column of Table~\ref{tab:RGD_ES_interpolant}. 
Our method is much more efficient and accurate and provides mappings with better bijectivity.

\begin{table}[htbp]
   \caption{Adaptive area-preserving parameterization for genus-0 closed surfaces proposed by Choi et al. \cite{CGK:2022} applied to the twelve mesh models considered in this work. Stopping criterion: default values.} 
   \label{tab:Choi_et_al}
   \center
   \small{
      \begin{tabular}{ccccc}
            \toprule
            Model Name       &   SD/Mean  &  $E_{A}(f)$   &  Time    & \#Foldings \\
            \hline
            Right Hand       &   18.3283  &  $ 4.84 \times 10^{3}$   &      218.03  &      672  \\
            David Head       &   0.0426   &  $ 2.27 \times 10^{-2}$  &      298.76  &        0  \\
            Cortical Surface &   0.6320   &  $ 1.14 \times 10^{0}$   &      420.20  &       10  \\
            Bull             &   8.5565   &  $ 1.82 \times 10^{3}$   &       34.42  &      335  \\
            Bulldog          &   9.2379   &  $ 1.22 \times 10^{3}$   &      183.94  &      338  \\
            Lion Statue      &   0.2626   &  $ 8.96 \times 10^{-1}$  &     1498.91  &      540  \\
            Gargoyle         &   0.3558   &  $ 1.30 \times 10^{0}$   &     1483.35  &      571  \\
            Max Planck       &   11.6875  &  $ 1.49 \times 10^{3}$   &      195.39  &      575  \\
            Bunny            &   27.6014  &  $ 8.94 \times 10^{3}$   &      157.87  &      208  \\
            Chess King       &   11.8300  &  $ 1.65 \times 10^{3}$   &      608.55  &      948  \\
            Art Statuette    &  394.4414  &  $ 9.93 \times 10^{0}$   &     2284.79  &     2242  \\
            Bimba Statue     &    0.5110  &  $ 2.01 \times 10^{0}$   &  16\,773.34  &  11\,821  \\
            \bottomrule
      \end{tabular}
      }
\end{table}

\subsection{Numerical stability}\label{sec:numerical_stability}

In this section, we investigate the numerical stability of our scheme. To this aim, we introduce Gaussian random noise to the vertex normal of each vertex in every mesh model according to a given value of noise variance $\sigma_{\mathrm{noise}}$. We then re-run the entire algorithm, i.e., we first perform a few iterations of the FPI method (Algorithm~\ref{algo:FPI_SEM}) to obtain a mapping which is used as an initial mapping for the RGD method (Algorithm~\ref{algo:RGD}). We then calculate a relative error on the authalic energy, as defined below in~\eqref{eq:err-EA}. We can also repeat this procedure for different values of noise variance $\sigma_{\mathrm{noise}}$.

Figure~\ref{fig:LionStatue_with_Noise} shows the original mesh model of the Lion Statue (panel (a)) and two noisy versions (panels (b) and (c)).

\begin{figure}[htbp]
   \centering
   \subfloat[]{\includegraphics[height=5cm]{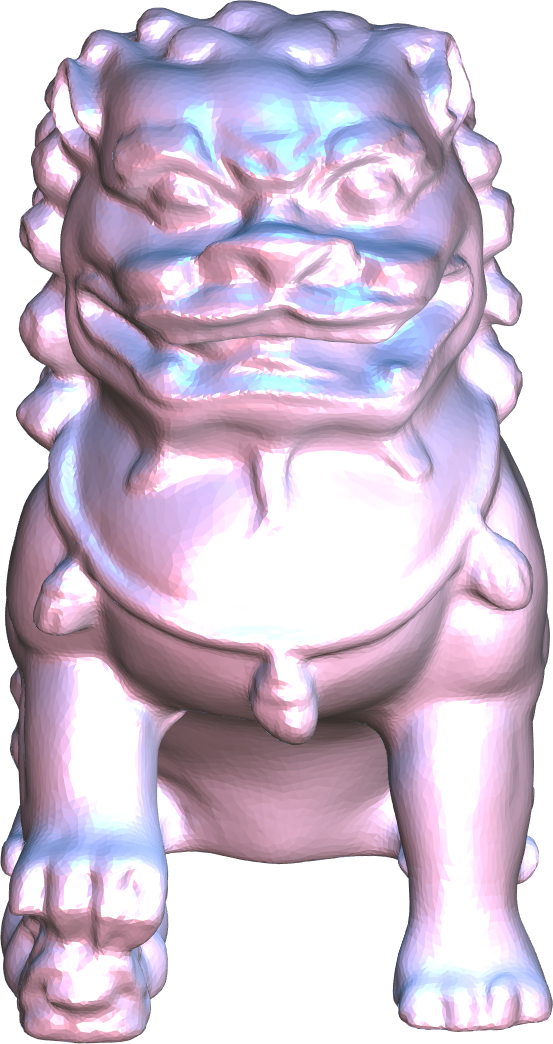}} \hspace{1cm}
   \subfloat[]{\includegraphics[height=5cm]{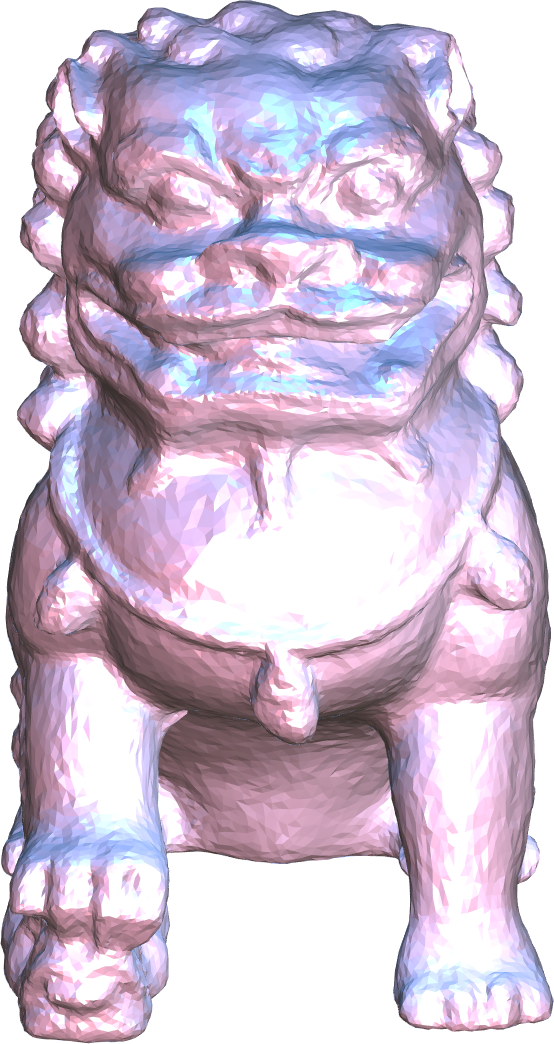}}  \hspace{1cm}
   \subfloat[]{\includegraphics[height=5cm]{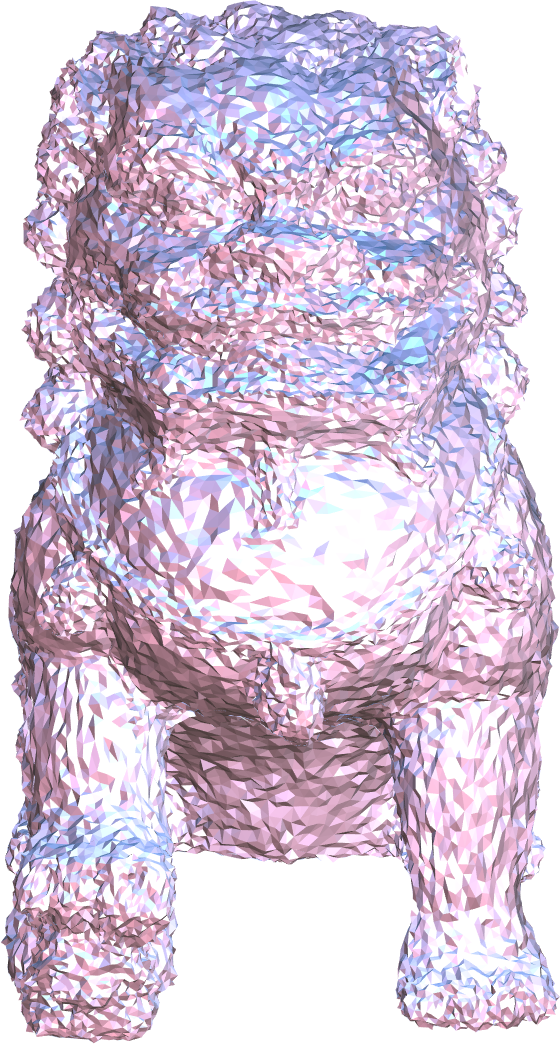}}
   \caption{The ``Lion Statue'' mesh model: (a) without noise; (b), (c) with noise variance $\sigma_{\mathrm{noise}} = 1 \times 10^{-3}$ and $5 \times 10^{-3}$, respectively.}
   \label{fig:LionStatue_with_Noise}
\end{figure}

We compute the following relative error on the authalic energy
\begin{equation}\label{eq:err-EA}
   \text{err-}E_{A}\big(f,\tilde{f}\big) \coloneqq \#\text{Vertices} \times \frac{\lvert {E}_{A}\big(\tilde{f}\big) - E_{A}(f) \rvert}{\sum_{v\in\mathcal{V}(\cM)} \| \widetilde{v} - v\|_2
   },
\end{equation}
where ${E}_{A}\big(\tilde{f}\big)$ and $E_{A}(f)$ are the authalic energies after 100 iterations of RGD for the mesh model with and without noise, respectively, and $\widetilde{v}$ and $v$ denote the coordinates of the vertices of the mesh model with and without noise, respectively.

Tables~\ref{tab:numerical_stability_1} and~\ref{tab:numerical_stability_2} report the numerical results for all the noisy mesh models, with $\sigma_{\mathrm{noise}} = 1 \times 10^{-3}$ and $5 \times 10^{-3}$, respectively. The values of authalic energy for the non-noisy mesh models from Table~\ref{tab:RGD_ES_interpolant} are reported in the second last column for easier comparison. We observe that the values for SD/Mean and $E_{A}\big(\tilde{f}\big)$ remain bounded and reasonable with respect to the original mesh models, which demonstrate that our method is stable to noise.

\begin{table}[htbp]
   \caption{Numerical stability study: 100 iterations of RGD with $\sigma_{\mathrm{noise}} = 1 \times 10^{-3} $. Line-search strategy: quadratic/cubic approximation from \protect{\cite[\S 6.3.2]{DennisSchnabel:1996}}. The quantity err-$E_{A}$ is defined in~\eqref{eq:err-EA}.} 
   \label{tab:numerical_stability_1}
   \center
   \resizebox{\textwidth}{!}{
      \begin{tabular}{@{} *{8}{c}}
            \toprule
            & \multicolumn{5}{c}{$\sigma_{\mathrm{noise}} = 1 \times 10^{-3} $} \\
            \cline{2-6}
            \multirow{2}{*}{Model Name}   &   \multirow{2}{*}{SD/Mean}  &  \multirow{2}{*}{$E_{A}\big(\tilde{f}\big)$}  &  \multirow{2}{*}{Time}  & \multirow{2}{*}{\#Fs}  &  \#Fs  & \multirow{2}{*}{$E_{A}(f)$} & \multirow{2}{*}{err-$E_{A}\big(f,\tilde{f}\big)$} \\
              &     &    &    &   &  {\footnotesize after b.c.}  &  \\
            \cmidrule(r){1-1}  \cmidrule(lr){2-6}  \cmidrule(lr){7-8}
            {\footnotesize Right Hand}       &   0.1254  &   $ 1.03 \times 10^{-1}$   &    3.13  &   14   &   1   &  $ 9.40 \times 10^{-2}$  &    $ 5.84 \times 10^{-2}$   \\
            {\footnotesize David Head}       &   0.0108  &   $ 1.44 \times 10^{-3}$   &    9.98  &    0   &   0   &  $ 3.04 \times 10^{-3}$  &    $ 4.47 \times 10^{-3}$   \\
            {\footnotesize Cortical Surface} &   0.0187  &   $ 3.82 \times 10^{-3}$   &   12.78  &    0   &   0   &  $ 3.72 \times 10^{-3}$  &    $ 3.45 \times 10^{-3}$   \\
            {\footnotesize Bull}             &   0.2528  &   $ 9.23 \times 10^{-1}$   &   16.74  &   19   &   6   &  $ 2.19 \times 10^{-1}$  &    $ 1.12 \times 10^{0} $   \\
            {\footnotesize Bulldog}          &   0.0273  &   $ 6.94 \times 10^{-3}$   &   59.26  &    0   &   0   &  $ 1.27 \times 10^{-2}$  &    $ 7.19 \times 10^{-6}$   \\
            {\footnotesize Lion Statue}      &   0.1630  &   $ 3.27 \times 10^{-1}$   &   66.78  &    1   &   0   &  $ 4.54 \times 10^{-1}$  &    $ 1.36 \times 10^{-4}$   \\
            {\footnotesize Gargoyle}         &   0.1677  &   $ 3.66 \times 10^{-1}$   &   65.72  &    0   &   0   &  $ 4.76 \times 10^{-2}$  &    $ 4.45 \times 10^{-3}$   \\
            {\footnotesize Max Planck}       &   0.0464  &   $ 2.67 \times 10^{-2}$   &   64.30  &    0   &   0   &  $ 3.39 \times 10^{-2}$  &    $ 5.27 \times 10^{-5}$   \\
            {\footnotesize Bunny}            &   0.0297  &   $ 1.11 \times 10^{-2}$   &   81.31  &    0   &   0   &  $ 1.91 \times 10^{-2}$  &    $ 5.32 \times 10^{-3}$   \\
            {\footnotesize Chess King}       &   0.0747  &   $ 6.08 \times 10^{-2}$   &  193.04  &  119   &  38   &  $ 5.23 \times 10^{-2}$  &    $ 1.03 \times 10^{-3}$   \\
            {\footnotesize Art Statuette}    &   0.0235  &   $ 6.88 \times 10^{-3}$   &  636.30  &    0   &   0   &  $ 2.10 \times 10^{-2}$  &    $ 2.54 \times 10^{-2}$   \\
            {\footnotesize Bimba Statue}     &   0.0541  &   $ 3.33 \times 10^{-2}$   &  759.36  &    2   &   0   &  $ 3.29 \times 10^{-2}$  &    $ 5.81 \times 10^{-4}$   \\
            \bottomrule
      \end{tabular}
    }
\end{table}

\begin{table}[htbp]
   \caption{Numerical stability study: 100 iterations of RGD with $\sigma_{\mathrm{noise}} = 5 \times 10^{-3} $. Line-search strategy: quadratic/cubic approximation from \protect{\cite[\S 6.3.2]{DennisSchnabel:1996}}. The quantity err-$E_{A}$ is defined in~\eqref{eq:err-EA}.} 
   \label{tab:numerical_stability_2}
   \center
   \resizebox{\textwidth}{!}{
      \begin{tabular}{@{} *{8}{c}}
            \toprule
            & \multicolumn{5}{c}{$\sigma_{\mathrm{noise}} = 5 \times 10^{-3} $} \\
            \cline{2-6}
            \multirow{2}{*}{Model Name}   &   \multirow{2}{*}{SD/Mean}  &  \multirow{2}{*}{$E_{A}\big(\tilde{f}\big)$}  &  \multirow{2}{*}{Time}  & \multirow{2}{*}{\#Fs}  &  \#Fs  & \multirow{2}{*}{$E_{A}(f)$} & \multirow{2}{*}{err-$E_{A}\big(f,\tilde{f}\big)$} \\
              &     &    &    &   &  {\footnotesize after b.c.}  &  \\
            \cmidrule(r){1-1}  \cmidrule(lr){2-6}  \cmidrule(lr){7-8}
            {\footnotesize Right Hand}       &   0.2304  &   $ 5.77 \times 10^{-1}$   &    3.12  &   10   &   4   &  $ 9.40 \times 10^{-2}$  &   $ 1.45 \times 10^{0}$   \\
            {\footnotesize David Head}       &   0.0211  &   $ 5.51 \times 10^{-3}$   &    9.43  &    0   &   0   &  $ 3.04 \times 10^{-3}$  &   $ 6.84 \times 10^{-3}$   \\
            {\footnotesize Cortical Surface} &   0.0301  &   $ 1.02 \times 10^{-2}$   &   12.72  &    0   &   0   &  $ 3.72 \times 10^{-3}$  &   $ 2.28 \times 10^{-1}$   \\        
            {\footnotesize Bull}             &   0.1723  &   $ 2.98 \times 10^{-1}$   &   17.40  &   18   &   4   &  $ 2.19 \times 10^{-1}$  &   $ 1.19 \times 10^{-1}$   \\
            {\footnotesize Bulldog}          &   0.0629  &   $ 3.86 \times 10^{-2}$   &   61.75  &    2   &   0   &  $ 1.27 \times 10^{-2}$  &   $ 3.23 \times 10^{-5}$   \\
            {\footnotesize Lion Statue}      &   0.1044  &   $ 1.36 \times 10^{-1}$   &   69.55  &    0   &   0   &  $ 4.54 \times 10^{-1}$  &   $ 3.36 \times 10^{-4}$   \\
            {\footnotesize Gargoyle}         &   0.0974  &   $ 9.68 \times 10^{-2}$   &   69.32  &    1   &   0   &  $ 4.76 \times 10^{-2}$  &   $ 6.87 \times 10^{-4}$   \\
            {\footnotesize Max Planck}       &   0.0782  &   $ 7.11 \times 10^{-2}$   &   67.30  &    0   &   0   &  $ 3.39 \times 10^{-2}$  &   $ 2.75 \times 10^{-4}$   \\
            {\footnotesize Bunny}            &   0.0455  &   $ 2.57 \times 10^{-2}$   &   80.93  &    0   &   0   &  $ 1.91 \times 10^{-2}$  &   $ 4.40 \times 10^{-3}$   \\
            {\footnotesize Chess King}       &   0.1521  &   $ 2.72 \times 10^{-1}$   &  187.57  &   95   &  35   &  $ 5.23 \times 10^{-2}$  &   $ 2.64 \times 10^{-2}$   \\
            {\footnotesize Art Statuette}    &   0.0981  &   $ 5.41 \times 10^{-2}$   &  617.75  &    0   &   0   &  $ 2.10 \times 10^{-2}$  &   $ 5.92 \times 10^{-2}$   \\
            {\footnotesize Bimba Statue}     &   0.1083  &   $ 1.25 \times 10^{-1}$   &  743.00  &   76   &   0   &  $ 3.29 \times 10^{-2}$  &   $ 1.68 \times 10^{-1}$   \\
            \bottomrule
      \end{tabular}
   }
\end{table}

\subsection{Registration problem between two brain surfaces}\label{sec:brain_registration}

A registration mapping between surfaces $\cM_0$ and $\cM_1$ refers to a bijective mapping $g \colon \cM_0\to\cM_1$. An ideal registration mapping keeps important landmarks aligned while preserving specified geometry properties. In this section, we demonstrate a framework for the computation of landmark-aligned area-preserving parameterizations of genus-zero closed surfaces. 

Suppose a set of landmark pairs $\{(p_i, q_i) \mid  p_i\in \cM_0, \ q_i\in\cM_1\}_{i=1}^m$ is given. The goal is to compute an area-preserving simplicial mapping $g\colon \cM_0\to\cM_1$ that satisfies $g(p_i) \approx q_i$, for $i=1, \ldots, m$. 
First, we compute area-preserving parameterizations $f_0\colon \cM_0\to\Stwo$ and $f_1\colon \cM_1\to\Stwo$ of surfaces $\cM_0$ and $\cM_1$, respectively. 
The simplicial registration mapping $h \colon \Stwo \to \Stwo $ that satisfies $h\circ f_0(p_i) = f_1(q_i)$, for $i=1, \ldots, m$, can be carried out by minimizing the registration energy
$$
E_R(h) = E_S(h) + \sum_{i=1}^m \lambda_i \| h\circ f_0(p_i) - f_1(q_i) \|^2.
$$
Let 
\[
   \mathbf{h} = 
   \begin{bmatrix}
      (h\circ f_0(v_1))\tr \\
      \vdots \\
      (h\circ f_0(v_n))\tr
   \end{bmatrix}
   =\begin{bmatrix}
      \mathbf{h}^1 & \mathbf{h}^2 & \mathbf{h}^3
   \end{bmatrix}
\]
be the matrix representation of $h$. 
The gradient with respect to $\mathbf{h}^s$ can be formulated as
\[
   \nabla_{\mathbf{h}^s} E_R(h) = 2 L_S(h)\, \mathbf{h}^s + \mathbf{r}^s,
\]
where $\mathbf{r}$ is the matrix of the same size as $\mathbf{h}$ given by
\[
   \mathbf{r}(i,:) =
   \begin{cases}
      2\lambda_i\left(\mathbf{h}(i,:) - (f_1(q_i))\tr\right) & \text{if $p_i$ is a landmark}, \\
      (0,0,0) & \text{otherwise}.
    \end{cases}
\]

In practice, we define the midpoints $c_i$ of each landmark pairs on $\Stwo$ as
\[
   c_i = \frac{1}{2}( f_0(p_i) + f_1(q_i)),
\]
for $i=1, \ldots, m$, and compute $h_0$ and $h_1$ on $\Stwo$ that satisfy $h_0\circ f_0(p_i) = c_i$ and $h_1\circ f_1(q_i) = c_i$, respectively. 
Ultimately, the registration mapping $g\colon \cM_0\to\cM_1$ is obtained by the composition mapping $g = f_1^{-1} \circ h_1^{-1} \circ h_0 \circ f_0$.
Figure~\ref{fig:scheme_surface_registration} schematizes this composition of functions for the landmark-aligned morphing process from one brain to another brain.

\begin{figure}[htbp]
\centering
  \begin{tikzpicture}[scale=1.2]
    \path     
      node (A) {$\cM_{0} $\hspace{2mm}\includegraphics[height=2cm]{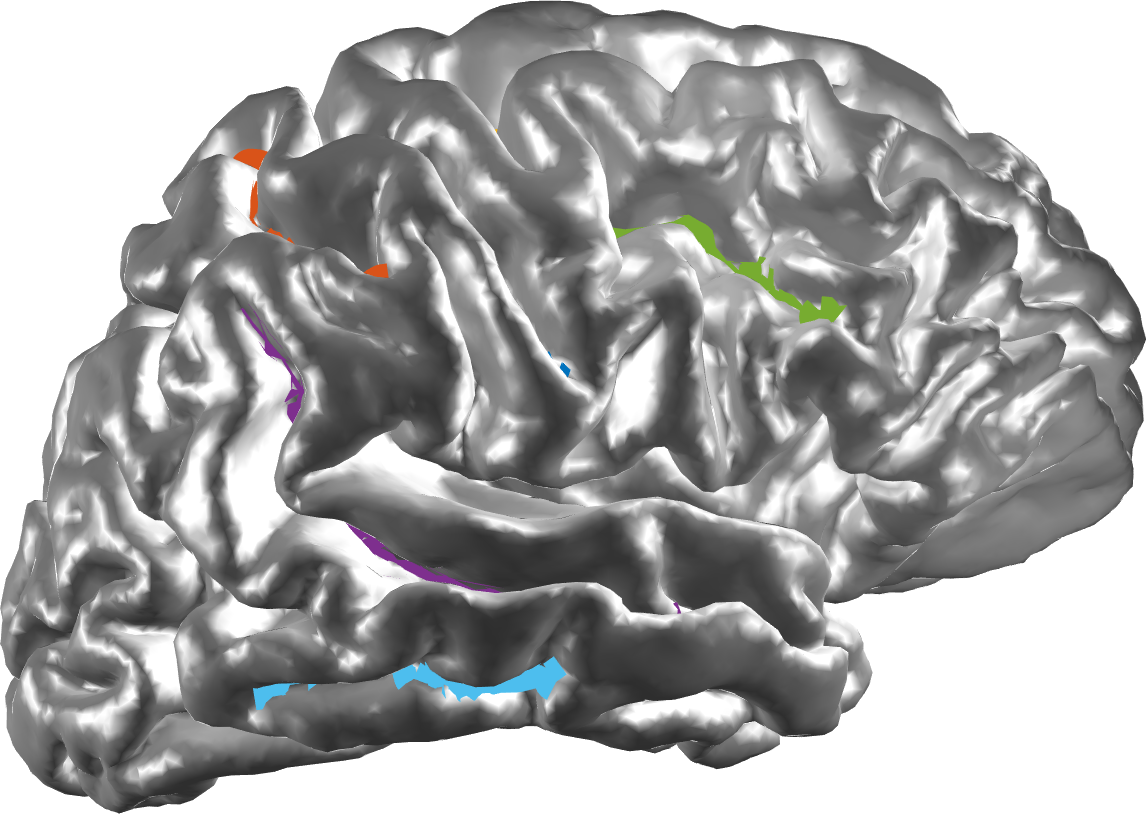}}
      (0:8cm)   node (B) {\includegraphics[height=2cm]{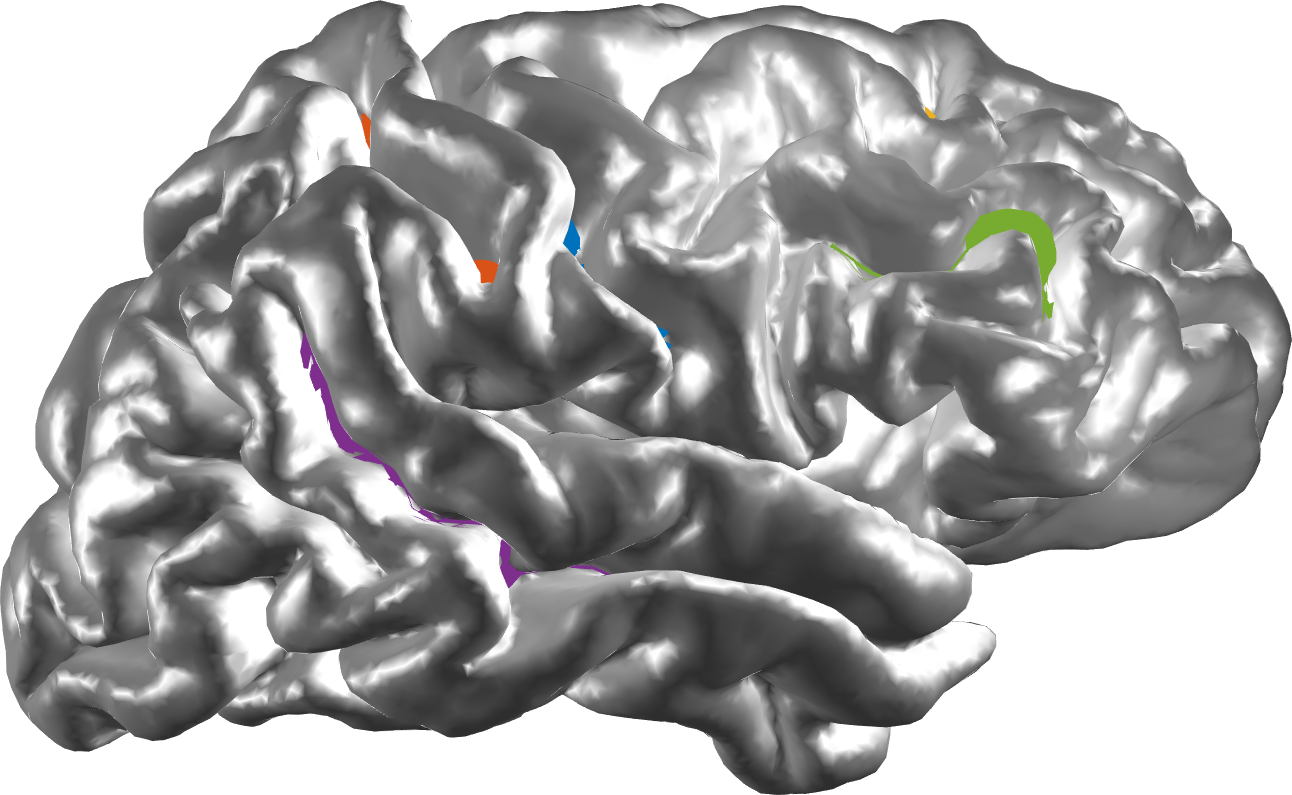}\hspace{-5mm}$\cM_{1}$}
      (-90:4cm) node (C) {$\Stwo$\includegraphics[height=2.2cm]{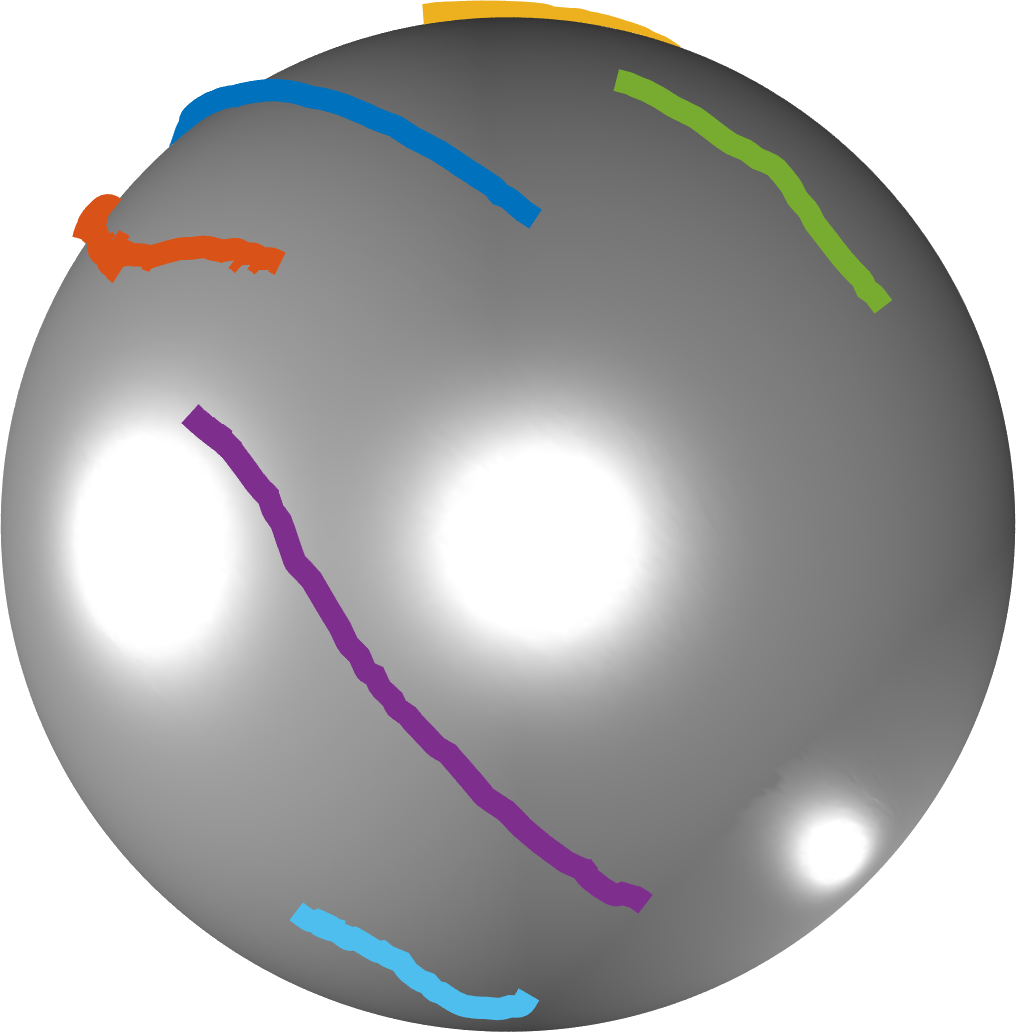}}
      (-26.5651:8.9442cm) node (D) {\includegraphics[height=2.2cm]{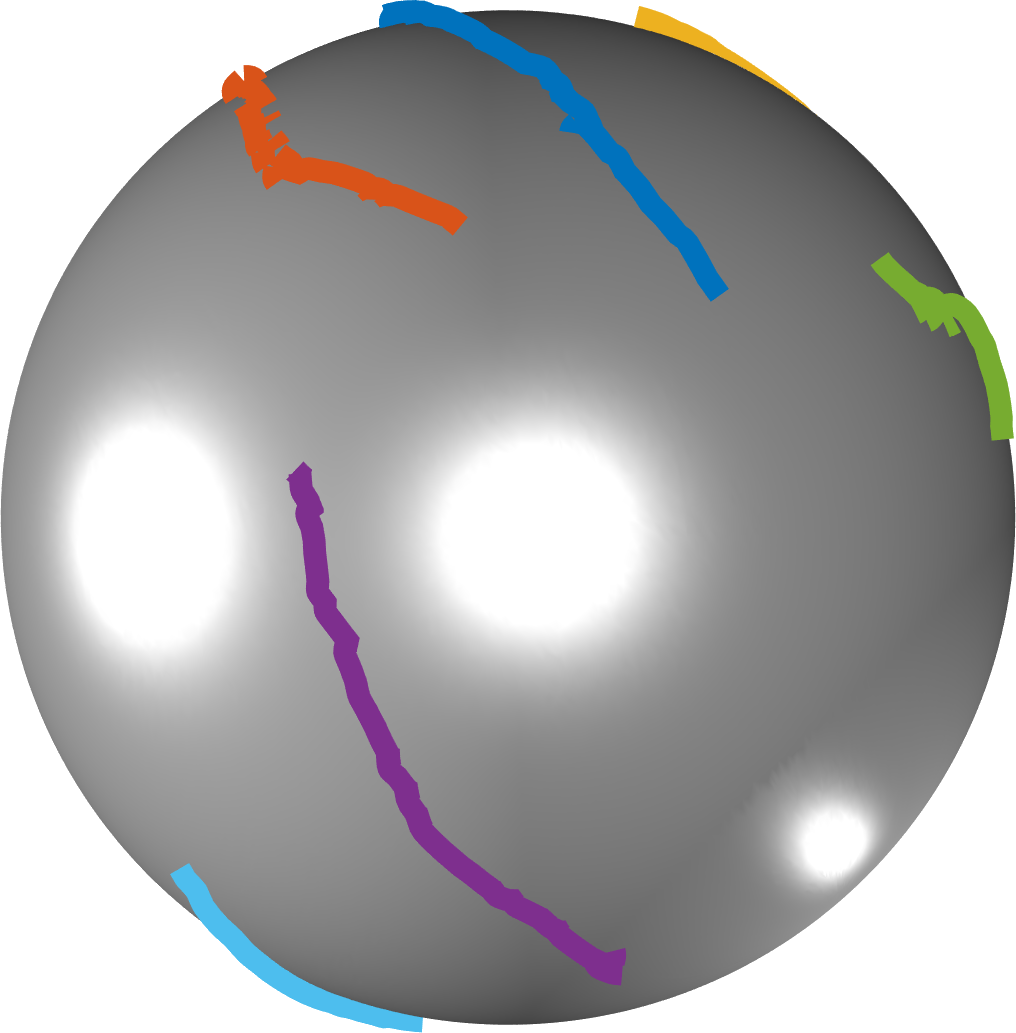}$\Stwo$}
      (-45:5.6568cm) node (E) {\includegraphics[height=2.2cm]{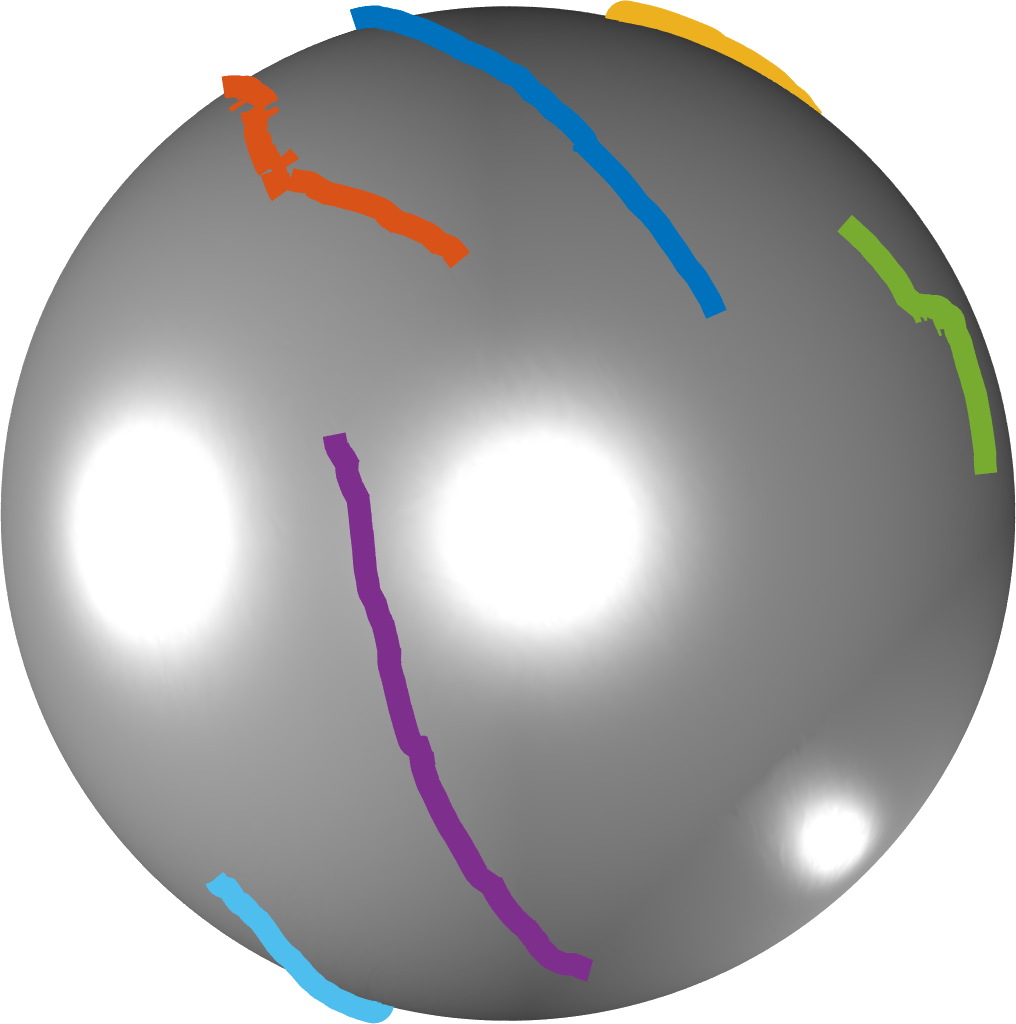}};
    \path[-To]
    (A) edge["{$g$}"]         node [right] {} (B)
    (A) edge["{$f_{0}$}"]     node [right] {} (C)
    (B) edge["{$f_{1}$}"]     node [right] {} (D)
    (C) edge["{$h_{0}$}"]     node [right] {} (E)
    (D) edge["{$h_{1}$}",style={sloped}]     node [right] {} (E);
    \end{tikzpicture}
   \caption{Scheme for the landmark-aligned surface registration application described in section \ref{sec:brain_registration}.}
\label{fig:scheme_surface_registration}
\end{figure}

A landmark-aligned morphing process from $\cM_0$ to $\cM_1$ can be constructed by the linear homotopy $H\colon \cM_0\times[0,1]\to\mathbb{R}^3$ defined as
\begin{equation} \label{eq:homotopy}
H(v,t) = (1-t) v + t g(v).
\end{equation}
In Figure \ref{fig:Homotopy}, we demonstrate the morphing process from one brain to another brain by four snapshots at four different values of $t$.
The brain surfaces are obtained from the source code package of \cite{CLL:2015}. 

\begin{figure}[htbp]
   \centering
   \resizebox{\textwidth}{!}{
   \begin{tabular}{cccc}
      \includegraphics[width=4cm]{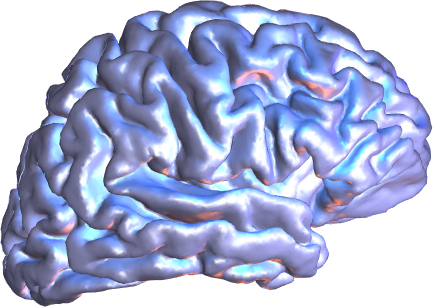} &
      \includegraphics[width=4cm]{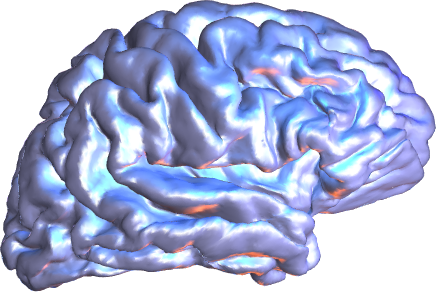} &
      \includegraphics[width=4cm]{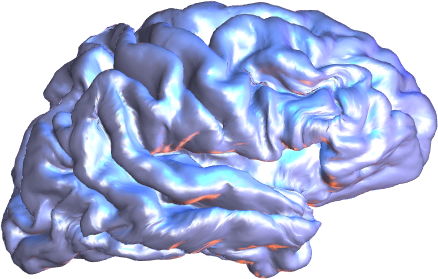} &
      \includegraphics[width=4cm]{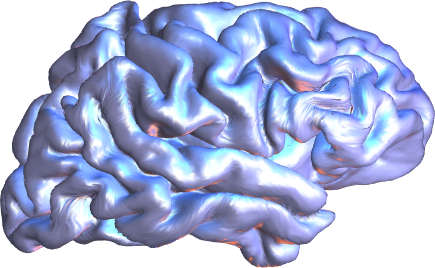} \\
      $H(\cM_0,0) = \cM_0$ & $H(\cM_0,0.33)$ & $H(\cM_0,0.67)$ & $H(\cM_0,1) = \cM_1$ 
   \end{tabular}
   }
   \caption{The images of the linear homotopy \eqref{eq:homotopy} from one brain to another brain.}
   \label{fig:Homotopy}
\end{figure}

\section{Conclusions and outlook}\label{sec:conclusions}

In this paper, we introduced an RGD method for computing spherical area-preserving mappings of genus-zero closed surfaces. Our approach combines the tools of Riemannian optimization and computational geometry to develop a method based on the minimization of the stretch energy. The proposed algorithm has theoretically guaranteed convergence and is accurate, efficient, and robust. We tested two different line-search strategies and conducted extensive numerical experiments on various mesh models to demonstrate the algorithm's stability and effectiveness. By comparing with two existing methods for computing area-preserving mappings, we demonstrated that our algorithm is more efficient than these state-of-the-art methods. Moreover, we show that our approach is stable and robust even when the mesh model undergoes small perturbations. Finally, we applied our algorithm to the practical problem of landmark-aligned surface registration between two human brain models.

There are some directions in which we could conduct further research. Specifically, we would like to enhance the speed of convergence of the algorithm we have proposed while keeping the computational cost low. One potential way to improve on this would be to use appropriate Riemannian generalizations of the conjugate gradient method or the limited memory BFGS (L-BFGS) method, as suggested in \cite{RingWirth:2012}.
Another research direction may target genus-one or higher genus closed surfaces.

\section*{Acknowledgments}

The work of the first author was supported by the National Center for Theoretical Sciences in Taiwan (R.O.C.) under the NSTC grant 112-2124-M-002-009-.
The work of the second author was supported by the National Science and Technology Council and the National Center for Theoretical Sciences in Taiwan.

\appendix

\section{Calculation of the gradient of the image area}\label{app:calculation_area}

The image area of the simplicial mapping $f$ can be calculated as follows. 
Let $\tau = [v_i,v_j,v_k]$ and denote $f_{ij}^\ell =f_{i}^\ell - f_{j}^\ell$, $f_{ik}^\ell =f_{i}^\ell - f_{k}^\ell$, and $f_{jk}^\ell =f_{ij}^\ell - f_{k}^\ell$, for $\ell=1,2,3$. 
The image area of a simplicial mapping $f$ can be formulated as
\[
\cA(f) = \sum_{\tau\in\cF(\cM)} |f(\tau)| 
= \sum_{\tau\in\cF(\cM)} \frac{1}{2} \sqrt{ 
\cA_{12}(f|_\tau)^2 + \cA_{13}(f|_\tau)^2 + \cA_{23}(f|_\tau)^2},
\]
where
$$
\cA_{12}(f|_{\tau}) = f_{ij}^1 f_{ik}^2 - f_{ij}^2 f_{ik}^1, \quad
\cA_{13}(f|_{\tau}) = f_{ij}^1 f_{ik}^3 - f_{ij}^3 f_{ik}^1, \quad
\cA_{23}(f|_{\tau}) = f_{ij}^2 f_{ik}^3 - f_{ij}^3 f_{ik}^2.
$$
The functionals $\cA_{12}$, $\cA_{13}$ and $\cA_{23}$ measure the image area of mappings $\Pi_{P_{xy}}\circ f$, $\Pi_{P_{xz}}\circ f$ and $\Pi_{P_{yz}}\circ f$, respectively.
The partial derivatives of $\cA(f|_{\tau})$ can be formulated as
\begin{align*}
\frac{\partial}{\partial f_i^1}\cA(f|_{\tau}) 
&= \frac{ 1}{4\cA(f|_{\tau})} \left( \cA_{12}(f|_{\tau}) f_{jk}^2 + \cA_{13}(f|_{\tau}) f_{jk}^3 \right), \\
\frac{\partial}{\partial f_i^2}\cA(f|_{\tau})
&= \frac{-1}{4\cA(f|_{\tau})} \left( \cA_{12}(f|_{\tau}) f_{jk}^1 - \cA_{23}(f|_{\tau}) f_{jk}^3 \right), \\
\frac{\partial}{\partial f_i^3}\cA(f|_{\tau}) 
&= \frac{-1}{4\cA(f|_{\tau})} \left( \cA_{13}(f|_{\tau}) f_{jk}^1 - \cA_{23}(f|_{\tau}) f_{jk}^2 \right), \\
\frac{\partial}{\partial f_j^1}\cA(f|_{\tau}) 
&= \frac{-1}{4\cA(f|_{\tau})} \left( \cA_{12}(f|_{\tau}) f_{ik}^2 + \cA_{13}(f|_{\tau}) f_{ik}^3 \right), \\
\frac{\partial}{\partial f_j^2}\cA(f|_{\tau})
&= \frac{ 1}{4\cA(f|_{\tau})} \left( \cA_{12}(f|_{\tau}) f_{ik}^1 - \cA_{23}(f|_{\tau}) f_{ik}^3 \right), \\
\frac{\partial}{\partial f_j^3}\cA(f|_{\tau})
&= \frac{ 1}{4\cA(f|_{\tau})} \left( \cA_{13}(f|_{\tau}) f_{ik}^1 + \cA_{23}(f|_{\tau}) f_{ik}^2 \right), \\
\frac{\partial}{\partial f_k^1}\cA(f|_{\tau})
&= \frac{ 1}{4\cA(f|_{\tau})} \left( \cA_{12}(f|_{\tau}) f_{ij}^2 + \cA_{13}(f|_{\tau}) f_{ij}^3 \right), \\
\frac{\partial}{\partial f_k^2}\cA(f|_{\tau})
&= \frac{-1}{4\cA(f|_{\tau})} \left( \cA_{12}(f|_{\tau}) f_{ij}^1 - \cA_{23}(f|_{\tau}) f_{ij}^3 \right), \\
\frac{\partial}{\partial f_k^3}\cA(f|_{\tau})
&= \frac{-1}{4\cA(f|_{\tau})} \left( \cA_{13}(f|_{\tau}) f_{ij}^1 + \cA_{23}(f|_{\tau}) f_{ij}^2 \right). 
\end{align*}

\section{Calculation of the Hessian of the stretch energy functional}

In this appendix, we describe the calculation of the Hessian of the stretch energy functional.
Let $\tau = [v_i, v_j, v_k]$. From \cite[Theorem 3.5 and (3.1)]{Yueh:2023}, we know that
\begin{align*}
\nabla_{\mathbf{f}^\ell} E_S(f|_\tau) 
= 2L_S(f|_\tau) \, \mathbf{f}|_\tau^\ell, ~ \ell = 1,2,3, 
\end{align*}
where
\begin{align*}
L_S(f|_\tau) = \frac{1}{4|\tau |} \begin{bmatrix} \begin{array}{c} (\mathbf{f}_i-\mathbf{f}_k)^\top (\mathbf{f}_j-\mathbf{f}_k)\\+(\mathbf{f}_i-\mathbf{f}_j)^\top (\mathbf{f}_k-\mathbf{f}_j) \end{array} & -(\mathbf{f}_i-\mathbf{f}_k)^\top (\mathbf{f}_j-\mathbf{f}_k) & - (\mathbf{f}_i-\mathbf{f}_j)^\top (\mathbf{f}_k-\mathbf{f}_j) \\[0.5cm]
-(\mathbf{f}_i-\mathbf{f}_k)^\top (\mathbf{f}_j-\mathbf{f}_k) & \begin{array}{c} (\mathbf{f}_i-\mathbf{f}_k)^\top (\mathbf{f}_j-\mathbf{f}_k) \\ +(\mathbf{f}_j-\mathbf{f}_i)^\top (\mathbf{f}_k-\mathbf{f}_i) \end{array} & -(\mathbf{f}_j-\mathbf{f}_i)^\top (\mathbf{f}_k-\mathbf{f}_i) \\[0.5cm]
-(\mathbf{f}_i-\mathbf{f}_j)^\top (\mathbf{f}_k-\mathbf{f}_j) & -(\mathbf{f}_j-\mathbf{f}_i)^\top (\mathbf{f}_k-\mathbf{f}_i) & \begin{array}{c} (\mathbf{f}_i-\mathbf{f}_j)^\top (\mathbf{f}_k-\mathbf{f}_j) \\
+ (\mathbf{f}_j-\mathbf{f}_i)^\top (\mathbf{f}_k-\mathbf{f}_i) \end{array} \end{bmatrix},
\end{align*}
and $\mathbf{f}|_\tau^\ell = (\mathbf{f}_i^\ell, \mathbf{f}_j^\ell, \mathbf{f}_k^\ell)^\top$, $\ell = 1,2,3$. 
A direct calculation yields that the Hessian matrix 
$$
\mathrm{Hess}(E_S(f|_{\tau})) = \begin{bmatrix}
\frac{\partial E_S(f|_\tau)}{\partial f_i^s \partial f_i^t} &
\frac{\partial E_S(f|_\tau)}{\partial f_i^s \partial f_j^t} &
\frac{\partial E_S(f|_\tau)}{\partial f_i^s \partial f_k^t} \\[0.5cm]
\frac{\partial E_S(f|_\tau)}{\partial f_j^s \partial f_i^t} &
\frac{\partial E_S(f|_\tau)}{\partial f_j^s \partial f_j^t} &
\frac{\partial E_S(f|_\tau)}{\partial f_j^s \partial f_k^t} \\[0.5cm]
\frac{\partial E_S(f|_\tau)}{\partial f_k^s \partial f_i^t} &
\frac{\partial E_S(f|_\tau)}{\partial f_k^s \partial f_j^t} &
\frac{\partial E_S(f|_\tau)}{\partial f_k^s \partial f_k^t} 
\end{bmatrix}_{s,t=1}^3 \in\mathbb{R}^{9\times 9}
$$
can be formulated as
{\small
$$
\mathrm{Hess}(E_S(f|_{\tau})) = \frac{1}{2|\tau|}
\begin{bmatrix}
\mathbf{h}_{ijk}^2 {\mathbf{h}_{ijk}^2}\tr + \mathbf{h}_{ijk}^3 {\mathbf{h}_{ijk}^3}\tr &
\mathbf{h}_{ijk}^1 {\mathbf{h}_{ijk}^2}\tr - 2\mathbf{h}_{ijk}^2 {\mathbf{h}_{ijk}^1}\tr &
\mathbf{h}_{ijk}^1 {\mathbf{h}_{ijk}^3}\tr - 2\mathbf{h}_{ijk}^3 {\mathbf{h}_{ijk}^1}^\top \\[0.5cm]
\mathbf{h}_{ijk}^2 {\mathbf{h}_{ijk}^1}^\top - 2\mathbf{h}_{ijk}^1 {\mathbf{h}_{ijk}^2}^\top &
\mathbf{h}_{ijk}^1 {\mathbf{h}_{ijk}^1}^\top + \mathbf{h}_{ijk}^3 {\mathbf{h}_{ijk}^3}^\top &
\mathbf{h}_{ijk}^2 {\mathbf{h}_{ijk}^3}^\top - 2\mathbf{h}_{ijk}^3 {\mathbf{h}_{ijk}^2}^\top \\[0.5cm]
\mathbf{h}_{ijk}^3 {\mathbf{h}_{ijk}^1}^\top - 2\mathbf{h}_{ijk}^1 {\mathbf{h}_{ijk}^3}^\top &
\mathbf{h}_{ijk}^3 {\mathbf{h}_{ijk}^2}^\top - 2\mathbf{h}_{ijk}^2 {\mathbf{h}_{ijk}^3}^\top &
\mathbf{h}_{ijk}^1 {\mathbf{h}_{ijk}^1}^\top + \mathbf{h}_{ijk}^2 {\mathbf{h}_{ijk}^2}^\top
\end{bmatrix},
$$}
where $\mathbf{h}_{ijk}^\ell = (f_j^\ell-f_k^\ell, f_k^\ell-f_i^\ell, f_i^\ell-f_j^\ell)\tr$, $\ell=1,2,3$.

\section{Line-search procedure of the RGD method}\label{sec:line_search}

In this appendix, we describe the line-search procedure used in our RGD method. We start by detailing how to compute the derivative appearing in the sufficient decrease condition, and then we describe the interpolant line-search strategy from \protect{\cite[\S 6.3.2]{DennisSchnabel:1996}}.

In Euclidean space $\R^{n}$, the steepest descent method updates a current iterate by moving in the direction of the anti-gradient, by a step size chosen according to an appropriate line-search rule, $ \x^{(k)} + \alpha \, \bd^{(k)} $. Recall the sufficient decrease condition~\protect{\cite{NW:2006}}
\begin{equation}\label{eq:suff_decrease_cond}
   \phi(\alpha_{k}) \leq \phi(0) + c_{1} \alpha_{k} \phi'(0).
\end{equation}
In the Euclidean setting, the univariate function $\phi(\alpha)$ in the line-search procedure is
\[
   \phi(\alpha) = E\big(\f^{(k)} + \alpha \, \bd^{(k)}\big).
\]
However, this function changes in the  Riemannian optimization framework because we cannot directly perform the vector addition $\f^{(k)} + \alpha \, \bd^{(k)}$ without leaving the manifold.

Let $\f^{(k)}$ be a point of $ \Stwon $ at the $k$th iteration of our RGD algorithm, and $ \Retraction_{\f^{(k)}} $ the retraction at $\f^{(k)}$, as defined in section~\ref{sec:retraction_on_S2n}. Let the objective function be the $E \colon \Stwon \to \R$. For a fixed point $\f^{(k)}$ and a fixed tangent vector $ \bd^{(k)} $, we introduce the vector-valued function $ \boldsymbol{\psi} \colon \R \to \Stwon $, defined by $ \alpha \mapsto \Retraction_{\f^{(k)}}(\alpha\,\bd^{(k)}) $. Hence we have $ \phi \colon \R \to \R $, defined by
\[
   \phi(\alpha) \coloneqq E(\boldsymbol{\psi} (\alpha)).
\]
Note that $\phi(0) = E(\boldsymbol{\psi}(0)) = E(\f^{(k)}) \eqqcolon E^{(k)} $, due to the definition of retraction.

To evaluate the sufficient decrease condition \eqref{eq:suff_decrease_cond}, we need to calculate $\phi'(0)$, and in turn, to compute $\phi'(0)$, we need to calculate the derivative of the retraction $\Retraction_{\f^{(k)}} (\alpha \, \bd^{(k)}) $ with respect to $\alpha$. The derivative of $\phi(\alpha) $ is given by the chain rule
\[
   \phi'(\alpha) = \nabla E(\boldsymbol{\psi}(\alpha))\tr  \boldsymbol{\psi}'(\alpha),
\]
and then we evaluate it at $\alpha = 0$, i.e.,
\[
   \phi'(0) = \nabla E(\f^{(k)})\tr  \boldsymbol{\psi}'(0).
\]
In general, the retraction and the derivative $\boldsymbol{\psi}'(\alpha)$ depend on the choice of the manifold. The differentials of a retraction provide the so-called \emph{vector transports}, and the derivative of the retraction on the unit sphere appears in~\protect{\cite[\S 8.1.2]{AMS:2008}}.

In the following formulas, we omit the superscript $^{(k)}$ referring to the current iteration of RGD and assume that the line-search direction $\bd$ is also partitioned as the matrix $\f$; see~\eqref{eq:representative_matrix}.
Recalling the retraction on the power manifold $ \Stwon $ from~\eqref{eq:retraction_Stwon}, we can write $\boldsymbol{\psi}(\alpha)$ as 
\[
   \boldsymbol{\psi}(\alpha) = \Retraction_{\f}(\alpha \, \bd) =
   \begin{bmatrix}
       \dfrac{1}{\|\f_{1} + \alpha \, \bd_{1}\|_{2}}  &  &  &  \\
         &  \dfrac{1}{\|\f_{2} + \alpha \, \bd_{2}\|_{2}}  &  &  \\
         &  & \ddots & \\
         &  &  &  \dfrac{1}{\|\f_{n} + \alpha \, \bd_{n}\|_{2}}
   \end{bmatrix}
   \begin{bmatrix}
       \f_{1}\tr + \alpha \, \bd_{1}\tr \\
       \f_{2}\tr + \alpha \, \bd_{2}\tr \\
       \vdots \\
       \f_{n}\tr + \alpha \, \bd_{n}\tr
   \end{bmatrix}.
\]
The derivative $\boldsymbol{\psi}'(\alpha)$ can be computed via the formula (cf. \protect{\cite[Example 8.1.4]{AMS:2008}})
\[
    \boldsymbol{\psi}'(\alpha) =
    \begin{bmatrix}
        -\dfrac{(\f_{1} + \alpha \, \bd_{1})\tr \bd_{1}}{\|\f_{1} + \alpha \, \bd_{1}\|_{2}^{3}} \, (\f_{1} + \alpha \, \bd_{1})\tr  \\[12pt]
        -\dfrac{(\f_{2} + \alpha \, \bd_{2})\tr \bd_{2}}{\|\f_{2} + \alpha \, \bd_{2}\|_{2}^{3}} \, (\f_{2} + \alpha \, \bd_{2})\tr \\[12pt]
        \vdots \\[12pt]
        -\dfrac{(\f_{n} + \alpha \, \bd_{n})\tr \bd_{n}}{\|\f_{n} + \alpha \, \bd_{n}\|_{2}^{3}} \, (\f_{n} + \alpha \, \bd_{n})\tr
   \end{bmatrix} +
   \begin{bmatrix}
       \dfrac{\bd_{1}\tr}{\|\f_{1} + \alpha \, \bd_{1}\|_{2}} \\[12pt]
       \dfrac{\bd_{2}\tr}{\|\f_{2} + \alpha \, \bd_{2}\|_{2}} \\[12pt]
       \vdots \\[12pt]
       \dfrac{\bd_{n}\tr}{\|\f_{n} + \alpha \, \bd_{n}\|_{2}}
   \end{bmatrix},
\]
At $\alpha=0$, this simplifies into
\[
    \boldsymbol{\psi}'(0) =
   \begin{bmatrix}
       \dfrac{\bd_{1}\tr}{\|\f_{1} \|_{2}} \\[12pt]
       \dfrac{\bd_{2}\tr}{\|\f_{2} \|_{2}} \\[12pt]
       \vdots \\[12pt]
       \dfrac{\bd_{n}\tr}{\|\f_{n} \|_{2}}
   \end{bmatrix}
   -
    \begin{bmatrix}
        \dfrac{(\f_{1} )\tr \bd_{1}}{\|\f_{1} \|_{2}^{3}} \, \f_{1}\tr  \\[12pt]
        \dfrac{(\f_{2} )\tr \bd_{2}}{\|\f_{2} \|_{2}^{3}} \, \f_{2}\tr \\[12pt]
        \vdots \\[12pt]
        \dfrac{(\f_{n} )\tr \bd_{n}}{\|\f_{n} \|_{2}^{3}} \, \f_{n}\tr
   \end{bmatrix}.
\]
This is the value needed in order to evaluate the sufficient decrease condition~\eqref{eq:suff_decrease_cond}.

\subsection{Quadratic/cubic approximation}

At every step of our RGD algorithm, we want to satisfy the sufficient decrease condition~\eqref{eq:suff_decrease_cond}.
Here, we adopt the safeguarded quadratic/cubic approximation strategy described in~\protect{\cite[\S 6.3.2]{DennisSchnabel:1996}}.

In the following, we let $\alpha_{k}$ and $\alpha_{k-1}$ denote the step lengths used at iterations $k$ and $k-1$ of the optimization algorithm, respectively.
We denote the initial guess using $\alpha_{0}$. We suppose that the initial guess is given; alternatively, one can use \protect{\cite[(3.60)]{NW:2006}} as initial guess, i.e.,
\[
   \alpha_{0} = \frac{2(E^{(k)} - E^{(k-1)})}{\phi'(0)}.
\]

If $\alpha_{0}$ satisfies the sufficient decrease condition, i.e.,
\[
   \phi(\alpha_{0}) \leq \phi(0) + c_{1} \alpha_{0} \phi'(0),
\]
then $\alpha_{0}$ is accepted as step length, and we terminate the search. Otherwise, we build a quadratic approximation $\phi_{\mathrm{q}}(\alpha)$ of $ \phi(\alpha) $ using the information we have, that is, $\phi(0)$, $\phi'(0)$, and $\phi(\alpha_{0})$. The quadratic model is
\[
   \phi_{\mathrm{q}}(\alpha) = \left[\phi(\alpha_{0}) - \phi(0) -\phi'(0) \, \alpha_{0}\right]\alpha^{2} + \phi'(0)\,\alpha + \phi(0).
\]
The new trial value $\alpha_{1}$ is defined as the minimizer of this quadratic, i.e.,
\[
   \alpha_{1} = - \frac{\phi'(0) \, \alpha_{0}^{2}}{2 \left[\phi(\alpha_{0}) - \phi(0) -\phi'(0) \, \alpha_{0}\right]}.
\]
We terminate the search if the sufficient decrease condition is satisfied at $\alpha_{1}$, i.e.,
\[
   \phi(\alpha_{1}) \leq \phi(0) + c_{1} \alpha_{1} \phi'(0).
\]
Otherwise, we need to backtrack again. We now have four pieces of information about $\phi(\alpha)$, so it is desirable to use all of them. Hence, at this and any subsequent backtrack step during the current iteration of RGD, we use a cubic model $\phi_{\mathrm{c}}(\alpha)$ that interpolates the four pieces of information $\phi(0)$, $\phi'(0)$, and the last two values of $\phi(\alpha)$, and set $\alpha_{k}$ to the value of $\alpha$ at which $\phi_{\mathrm{c}}(\alpha)$ has its local minimizer.

Let $\alpha_{\mathrm{prev}}$ and $\alpha_{\mathrm{2prev}}$ be the last two previous values of $\alpha_{k}$ tried in the backtrack procedure. The cubic that fits $\phi(0)$, $\phi'(0)$, $\phi(\alpha_{\mathrm{prev}})$, and $\phi(\alpha_{\mathrm{2prev}})$ is
\[
   \phi_{\mathrm{c}}(\alpha) = a\alpha^{3} + b\alpha^{2} + \phi'(0)\,\alpha + \phi(0),
\]
where 
\[
   \begin{bmatrix}
       a \\
       b
   \end{bmatrix} =
   \dfrac{1}{\alpha_{\mathrm{prev}}-\alpha_{\mathrm{2prev}}}
   \begin{bmatrix}
       \dfrac{1}{\alpha_{\mathrm{prev}}^{2}}   &  \dfrac{-1}{\alpha_{\mathrm{2prev}}^{2}} \\[14pt]
       \dfrac{-\alpha_{\mathrm{2prev}}}{\alpha_{\mathrm{prev}}^{2}}  &  \dfrac{\alpha_{\mathrm{prev}}}{\alpha_{\mathrm{2prev}}^{2}}
   \end{bmatrix}
   \begin{bmatrix}
       \phi(\alpha_{\mathrm{prev}}) - \phi(0) - \phi'(0) \, \alpha_{\mathrm{prev}} \\[6pt] 
       \phi(\alpha_{\mathrm{2prev}}) - \phi(0) - \phi'(0) \, \alpha_{\mathrm{2prev}}
   \end{bmatrix}.
\]
The local minimizer of this cubic is given by \protect{\cite[(6.3.18)]{DennisSchnabel:1996}}
\[
   \frac{-b+\sqrt{b^{2}-3a\phi'(0)}}{3a},
\]
and we set $\alpha_{k}$ equal to this value.
If necessary, this process is repeated, using a cubic interpolant of $\phi(0)$, $\phi'(0)$, and the two most recent values of $\phi$, namely, $\phi(\alpha_{\mathrm{prev}})$ and $\phi(\alpha_{\mathrm{2prev}})$, until a step size that satisfies the sufficient decrease condition is located.
Numerical experiments in section~\ref{sec:numerical_experiments} demonstrate the usage of this line-search technique.

\bibliographystyle{aomalpha}

\begin{small}
    \bibliography{Sphere_area_preserving.bib}
\end{small}


\end{document}